\documentclass[11pt,a4paper]{amsart}
 \pdfoutput=1
\usepackage{amsfonts,amsmath,amssymb,amsthm}
\usepackage{times}
\usepackage{graphics,graphicx}
\usepackage{color}
\usepackage[notref,notcite]{}
\usepackage{subfigure}
\usepackage{pdfpages}

\numberwithin{equation}{section}

\renewcommand{\epsilon}{\varepsilon}

\DeclareSymbolFont{SY}{U}{psy}{m}{n}
\DeclareMathSymbol{\emptyset}{\mathord}{SY}{'306}

\DeclareMathOperator{\esup}{ess\,sup}

\newcommand{\R}{\mathbb{R}}

\newcommand{\N}{\mathbb{N}}

\newcommand{\cH}{{\mathcal H}}

\newcommand{\cO}{{\mathcal O}}

\newcommand{\cU}{{\mathcal U}}
\newcommand{\cV}{{\mathcal V}}

\newcommand{\cW}{{\mathcal W}}

\newcommand{\supp}{\mathrm{supp}}

{\bf}{\it}

\newtheorem{theorem}{Theorem}[section]{\bf}{\it}
{\bf}{\it}
\newtheorem{proposition}[theorem]{Proposition}{\bf}{\it}
\newtheorem{corollary}[theorem]{Corollary}{\bf}{\it}
{\it}{\rm}
\newtheorem{lemma}[theorem]{Lemma}{\bf}{\it}
\newtheorem{remark}[theorem]{Remark}{\it}{\rm}
\newtheorem{definition}[theorem]{Definition}{\bf}{\it}
{\bf}{\it}
{\bf}{\it}

{\bf}{\it}

{\bf}{\it}

\newtheorem{assumption}{Assumption}

\title[Hammertein fixed point problem]{Spatially localized solutions of the Hammerstein equation with sigmoid type of nonlinearity}

\author[A.~Oleynik]{Anna Oleynik}
\address{A.~Oleynik, Department of Mathematical Sciences and Technology, Norwegian University of Life Sciences, Postboks $5003$ NMBU
1432 {\AA}s}
\email{anna.oleynik@nmbu.no}

\author[A.~Ponosov]{Arcady Ponosov}
\address{A.~Ponosov, Department of Mathematical Sciences and Technology, Norwegian University of Life Sciences, Postboks $5003$ NMBU
1432 {\AA}s}
\email{arkadi@nmbu.no}

\author[V.~Kostrykin]{Vadim Kostrykin}
\address{V.~Kostrykin, FB 08 - Institut f\"{u}r Mathematik,
Johannes Gutenberg-Universit\"{a}t Mainz,
Staudinger Weg 9,
D-55099 Mainz,
Germany}
\email{kostrykin@mathematik.uni-mainz.de}

\author[A.V.~Sobolev]{Alexander V. Sobolev}
\address{A.V.~Sobolev, Department of Mathematics,
University College London, Gower Street, London WC1E 6BT, UK}
\email{a.sobolev@ucl.ac.uk}

\begin{document}

\maketitle

\begin{abstract}
We study the existence of fixed points to a parameterized Hammertstain operator $\cH_\beta,$ $\beta\in (0,\infty],$ with sigmoid type of nonlinearity. The parameter $\beta<\infty$ indicates the steepness of the slope of a nonlinear smooth sigmoid function and the limit case $\beta=\infty$ corresponds to a discontinuous unit step function. We prove that spatially localized solutions to the fixed point problem for large $\beta$ exist and can be approximated by the fixed points of $\cH_\infty.$ These results are of a high importance in biological applications where one often approximates the smooth sigmoid by discontinuous unit step function. Moreover, in order to achieve even better approximation than a solution of the limit problem, we employ the iterative method that has several advantages compared to other existing methods. For example, this method can be used to construct non-isolated homoclinic orbit of a Hamiltionian system of equations. We illustrate the results and advantages of the numerical method for stationary versions of the FitzHugh-Nagumo reaction-diffusion equation and a neural field model.

\end{abstract}

\section{Introduction}

We study the existence of solutions to the fixed point problem
\begin{equation} \label{eq:H}
u=\cH_\beta u, \quad(\cH_\beta u)(x):=\int \limits_\R \omega(x-y)f_\beta(u(y))dy.
\end{equation}

Here $\cH_\beta$ is the parameterized Hammerstein operator with $\beta \in (0,\infty],$ $\omega(x)$ is symmetric, and $f_\beta(u): \R \to [0,1]$ is a smooth function of sigmoid shape that approaches (in some way which we specify later) the unit step function $f_\infty=\chi_{[h,\infty)}$ for some $h>0$ as $\beta \to \infty.$ Examples of this type of function are

\begin{equation} \label{eq:f:sigmoid}
f_\beta(u)=S(\beta(u-h)), \quad S(u):=\dfrac{1}{1+\exp(-|u|)},
\end{equation}
and
\begin{equation} \label{eq:f:hill}
f_\beta(u)=S(\beta(u-h)), \quad S(u,p):=\dfrac{u^p}{u^p+1}\chi_{[0,\infty)}(u), \, p>0,
\end{equation}
see Fig. \ref{fig:Sigmas}.

\begin{figure}[h]
\scalebox{0.5}{\includegraphics{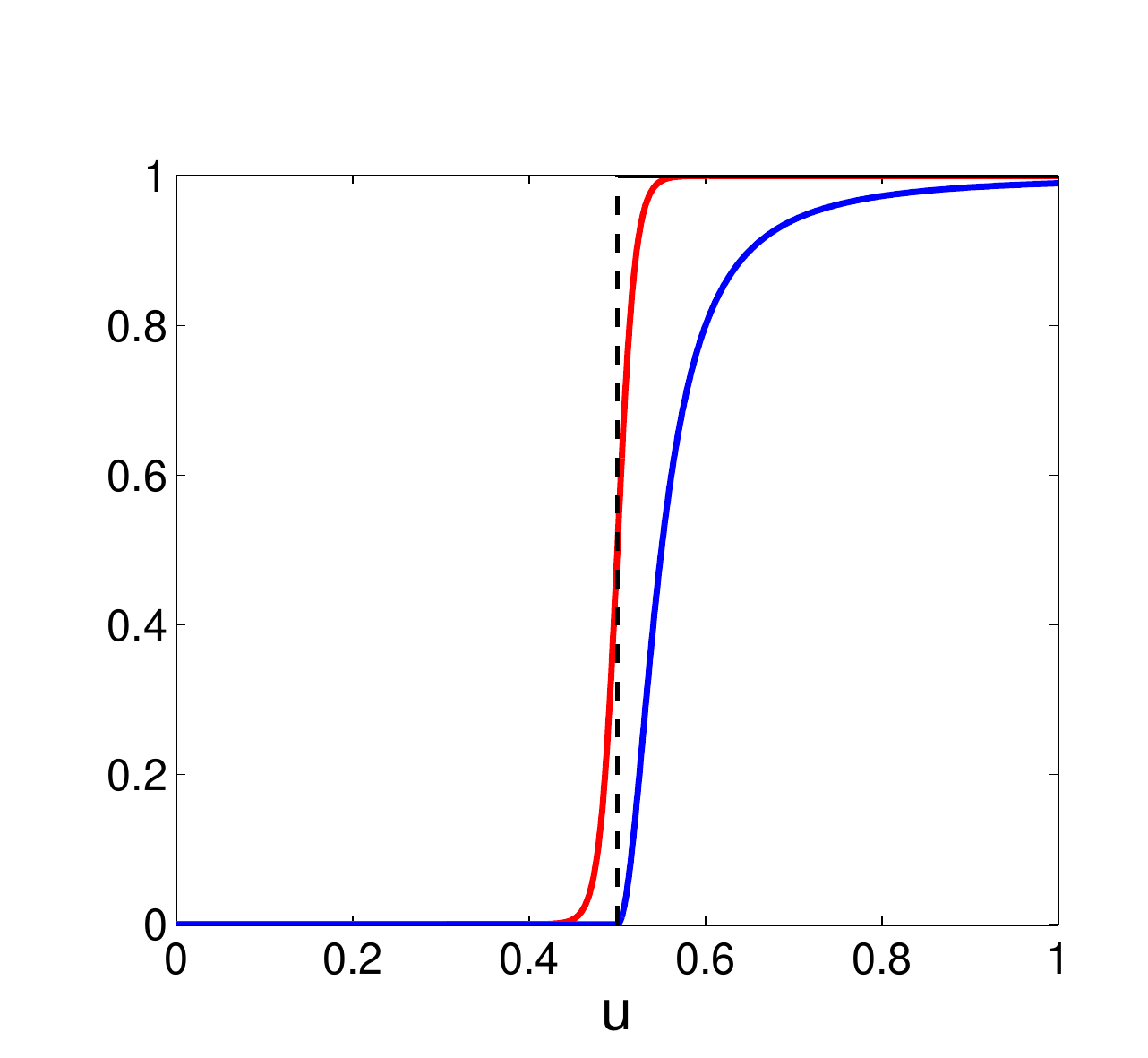}}
\caption{Functions $f_\beta$ given in \eqref{eq:f:sigmoid} with $\beta=100$ (red curve) and in \eqref{eq:f:hill} with $\beta=20$ (blue curve), and $f_\infty=\chi_{[h,\infty)}$ (black line), with $h=0.5.$ } \label{fig:Sigmas}
\end{figure}
This problem arises in several biological applications, e.g., in studying the existence of steady state solutions to a neural field model and the FitzHugh-Nagumo equation. We give examples of these models in Section \ref{sec:Examples}.

When the function $f_\beta$ is such that $f_\beta(u)=0$ for all $u<h,$ i.e., $\supp (f_\beta)\subset [h,\infty),$ all solutions to \eqref{eq:H} can be divided into two categories: (i) localized solutions and (ii) non-localized solutions (e.g. periodic, quasi-periodic). Here we study the solutions of the first class which we define in detail in Section \ref{sec:Bumps}.

For the limit case $\beta=\infty$, one often can construct localized solutions analytically, see e.g. \cite{Amari:1977} and Chapter $3$ in \cite{Coombes:NFM}. However, the case $\beta=\infty$ is only a simplification of a more realistic model where $f_\beta,$ $0<\beta<\infty$ is a steep yet smooth function. Analytical tools do not work in the latter case, and the existence of solutions for the case of $\beta<\infty$ and their continuous dependents on $\beta$ as $\beta \to 0,$ is often only conjectured from numerical simulations, see e.g. \cite{Coombes:2005,Avitabile:2014,McKean:1970}.

The main challenge of a proper justification of the transition between the cases $\beta<\infty$ and $\beta=\infty,$ is discontinuity of $f_\infty$ which leads to discontinuity of the corresponding integral operator in any standard functional space. To avoid this difficulty we suggest to exploit a spatial structure of localized solutions for $\beta=\infty$ and construct functional spaces such that
the operator $\cH_\infty$ is not only continuous but Fr\'{e}chet differentiable and the Implicit Function Theorem can be used. That is, we show that under the assumption that localized solution  of  \eqref{eq:H} for $\beta=\infty$ exists and satisfies some properties, solutions to \eqref{eq:H} for large $\beta<\infty$ exist, converge to a solution of the limit case as $\beta\to \infty$, and can be iteratively constructed.
We emphasize that this result is quite important as it provides a motivation for a common in applied sciences approximation of a smooth sigmoid function by the discontinuous unit step function.

 While the Implicit Function Theorem is well known to be a useful and powerful tool dealing with such problems, it has not been used due to the non-trivial choice of spaces for operator convergence. However there are similar results obtained in \cite{Oleynik:2013} and later extended by Burlakov et. al \cite{Burlakov:2015} for the case of a bounded spacial domain in $\R^n,$ $n\geq 1,$ using the topological degree theory and properties of the Hammerstain operators for a smaller class of solutions and more strict assumptions on the integral kernel $\omega.$

 The degree theory approach has some disadvantages compare to the method based on the Implicit Function Theorem proposed here. Firstly, it requires calculating the topological degree of $\cH_\infty$. Though it seems like the necessarily condition on the topological degree being non zero
 coincides with an assumption we introduce on the localized solution of \eqref{eq:H}, $\beta=\infty$, its calculation could be involved.
  Secondly, the degree theory never gives uniqueness for $\beta<\infty$ even if it is proven for $\beta=\infty.$
  Thirdly, as in infinite dimensional Banach spaces the degree is only defined for compact perturbations of the identity operator, see \cite{Zeidler:I}, the domain and the range space of $\cH_\beta$ must be the same. This restrict the choice of $\omega.$  Finally, the degree theory approach does not give a method for constructing solutions numerically.

The paper is organized as follows. In Section \ref{sec:Examples} we give examples of relevant applications, that is, the FitzHugh-Nagumo equation (Section \ref{sec:Examples:FN}) and the neural field model (Section \ref{sec:Examples:NFM}), that we use later to illustrate our results. We give a list of notations in Section \ref{sec:Notation}. In Section \ref{sec:Framework} we introduce assumptions, describe the problem and state the results. In particular, in Section \ref{sec:Properties} we describe some properties of the operator $\cH_\beta,$ in Section \ref{sec:Bumps} we give a definition of a bump and bump solution. The existence and assumptions on the solution to the problem \eqref{eq:H} for the limiting case $\beta=\infty$ is discussed in
Section \ref{sec:Limiting:problem}. We formulate our main results  as Theorem \ref{th:1} in Section \ref{sec:Main:results}. Section \ref{sec:Proof} is dedicated to the proof of Theorem \ref{th:1}. In Section \ref{sec:Conclusions} we apply our results to prove existence of stationary solutions to the FitzHugh-Nagumo equation (Section \ref{sec:Conclusions:FN}) and the neural field model (Section \ref{sec:Conclusions:NFM}), and numerically construct them. We also use these examples to discuss the advantages of the method in comparison to other existing methods. Section \ref{sec:Outlook} contains conclusions and outlook.

\section{Examples} \label{sec:Examples}
\subsection{FitzHugh-Nagumo equation and its caricature} \label{sec:Examples:FN}
The famous FitzHugh-Nagumo equation introduced in \cite{FitzHugh:1961,FitzHugh:1969,Nagumo:1962} describes qualitative properties of nerve conduction. The equation can be written as
\begin{equation} \label{eq:FitzHugh}
\begin{array}{l}
\dfrac{\partial u}{\partial t} =g(u)-v+D \dfrac{\partial^2 u}{\partial x^2}\\
\\
\dfrac{\partial v}{\partial t} = b u-\gamma v,
\end{array}
\end{equation}
where $g$ is commonly assumed to be the cubic function
\begin{equation} \label{eq:g:cubic}
g(u)=u(u-h)(1-u), \quad 0<h<1,
\end{equation}
and $b,$ $\gamma,$ and $D$ are positive parameters.

McKean  in \cite{McKean:1970} suggested replacing the cubic $g$ with a broken line of the same general shape, that is,
\begin{equation} \label{eq:g:step}
g(u)=-u+f_\infty(u), \quad f_\infty=\chi_{[h,\infty)}, \quad 0<h<1,
\end{equation}
see Fig. \ref{fig:g}.
The equation \eqref{eq:FitzHugh} with \eqref{eq:g:step} is commonly referred to as a caricature of the FitzHugh-Nagumo equation.
It is believed but not proven that the equation \eqref{eq:FitzHugh} with the cubic function \eqref{eq:g:cubic} and with \eqref{eq:g:step} have similar portraits in the large (and indeed they have similar phase portraits for the steady state solution equations for some parameter values). However, since discontinuous $g=-u+f_\infty(u)$ does not provide a good approximation of the cubic function \eqref{eq:g:cubic} but only resembles its shape, see Fig.\ref{fig:g}, it is doubtful that one can draw any rigourous conclusion about one model from analysing another. And indeed, the analysis of the equation \eqref{eq:FitzHugh} with the function \eqref{eq:g:cubic} and \eqref{eq:g:step} are rather disconnected in the literature.

\begin{figure}[h]
\scalebox{0.7}{\includegraphics{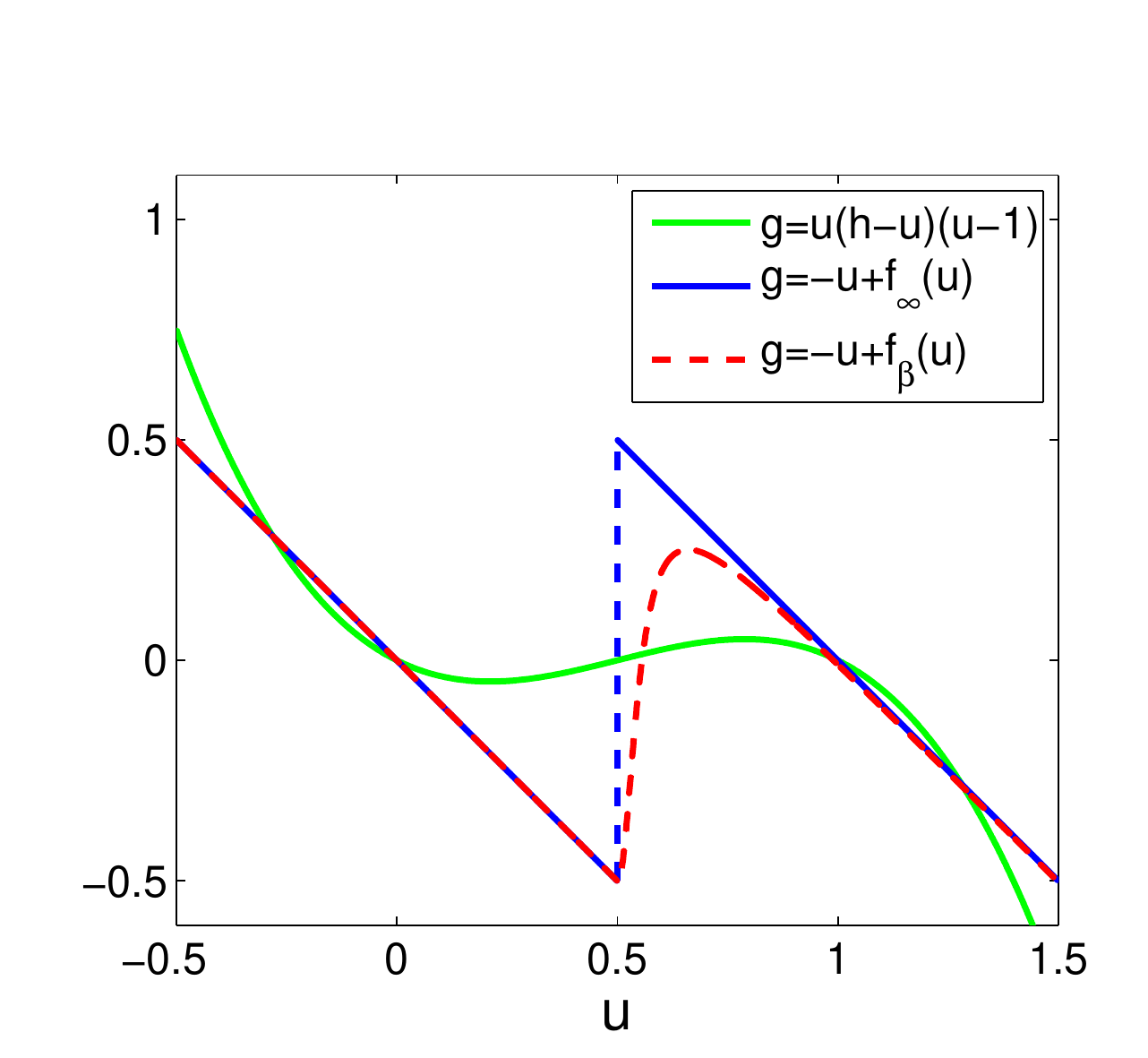}}
\caption{Functions $g=u(u-h)(1-u)$, $g=-u+f_\beta(u)$ with $f_\beta(u)$ given as in \eqref{eq:f:hill} with $\beta=20$ and $g=-u+f_\infty(u)$ given in \eqref{eq:g:step}, with $h=0.5.$ } \label{fig:g}
\end{figure}

However, a smooth sigmoidal shape function  $g=-u+ f_\beta(u)$ may provide a good approximation of $g=-u+f_\infty(u)$ for large $\beta,$ see Fig. \ref{fig:g}.  Thus, one would be able to draw a common conclusion about existence and behaviour of solutions obtained for either case. This leads us to study the reaction diffusion system

\begin{equation} \label{eq:FitzHugh:2}
\begin{array}{l}
\dfrac{\partial u}{\partial t} =f_\beta(u)-(u+v)+D \dfrac{\partial^2 u}{\partial x^2}\\
\\
\dfrac{\partial v}{\partial t} = b u-\gamma v.
\end{array}
\end{equation}
with $\beta \in (0,\infty].$

The steady states of \eqref{eq:FitzHugh:2} are solutions to
\begin{subequations}\label{eq:FN:steady}
\begin{align}
-u'' + k^2 u=f_\beta(u)\label{eq:FN:steady:a}
\\
v=\dfrac{b}{\gamma}u,\label{eq:FN:steady:b}
\end{align}
\end{subequations}
where  $k=\sqrt{1+{b}/{\gamma}}.$

Any bounded solution of the ordinary differential equation  in \eqref{eq:FN:steady:a} also satisfies the integral equation \eqref{eq:H}, i.e.,
\begin{equation} \label{eq:FN:steady:int}
u(x)=\int\limits_\R \omega(x-y)f_\beta(u(y)),
\end{equation}
with
\begin{equation}\label{eq:omega:exp}
\omega(x)=\frac{1}{2k}e^{-k|x|},
\end{equation}
see Fig. \ref{fig:omegas}.

Here \eqref{eq:omega:exp} is the Green function to the linear part of the equation in \eqref{eq:FN:steady:a} which can be shown in a similar way as in \cite{Krisner:2004}.

\subsection{Neural field model}\label{sec:Examples:NFM}
The behavior of a single layer of neurons can be modeled by a nonlinear integro-differential equation of the Hammerstein type,
\begin{equation}\label{eq:1}
 \partial_t u(x,t) = -u(x,t) + \int_{\R} \omega(x-y) f(u(y,t)) dy.
\end{equation}
Here $u(x,t)$ and $f(u(x,t))$ represent the averaged local activity and the firing rate of neurons at the position $x\in \R$ and time $t>0$, respectively, and $\omega(x-y)$ describes a coupling between neurons at positions $x$ and $y$.

The model above is often referred to as the Amari model and is a version of neural field models that constitute a special class of models where the neural tissue is treated as a continuous structure. The model  \eqref{eq:1} has been studied in numerous mathematical papers, for a review see, e.g., \cite{Coombes:2005,Ermentrout:1998} and \cite{Coombes:NFM}. In particular, the global existence and uniqueness of solutions to the initial value problem for \eqref{eq:1} under rather mild assumptions on $f$ and $\omega$ has been proven in \cite{Potthast:Graben:2010}.

In 1977, Amari studied pattern formation in \eqref{eq:1} for a model where $f$ is the unit step function and
$\omega$ is assumed to be of the lateral-inhibitory type, i.e., continuous, integrable and even, with $\omega(0)>0$ and having exactly one positive zero. In particular, he showed the existence of stable and unstable time independent spatially localized solutions to \eqref{eq:1} which he referred to as bumps. For more general $f$ and $\omega$ the existence of  solutions of this kind has been shown by Kishimoto and Amari in \cite{Kishimoto:Amari:1979} and later generalized in \cite{Oleynik:2015} and \cite{Kostrykin:Oleynik:2013}.

In what follows we use the Amari terminology, i.e., spatially localized solutions will be called  bumps.

Since the work by Amari the lateral-inhibitory type of connectivity function is the common choice when one studies pattern formation in neural field models. Examples of this type of connectivity are

\begin{equation} \label{eq:omega:wizard}
\omega(x)=(1-|x|) e^{-k|x|}, \quad k>0
\end{equation}
and
\begin{equation} \label{eq:omega:2Gaussians}
\omega(x)=K e^{-(kx)^2}-M e^{-(mx)^2}, \quad K>M, \, k>m,
\end{equation}
see Fig. \ref{fig:omegas}.

\begin{figure}[h]
\scalebox{0.7}{\includegraphics{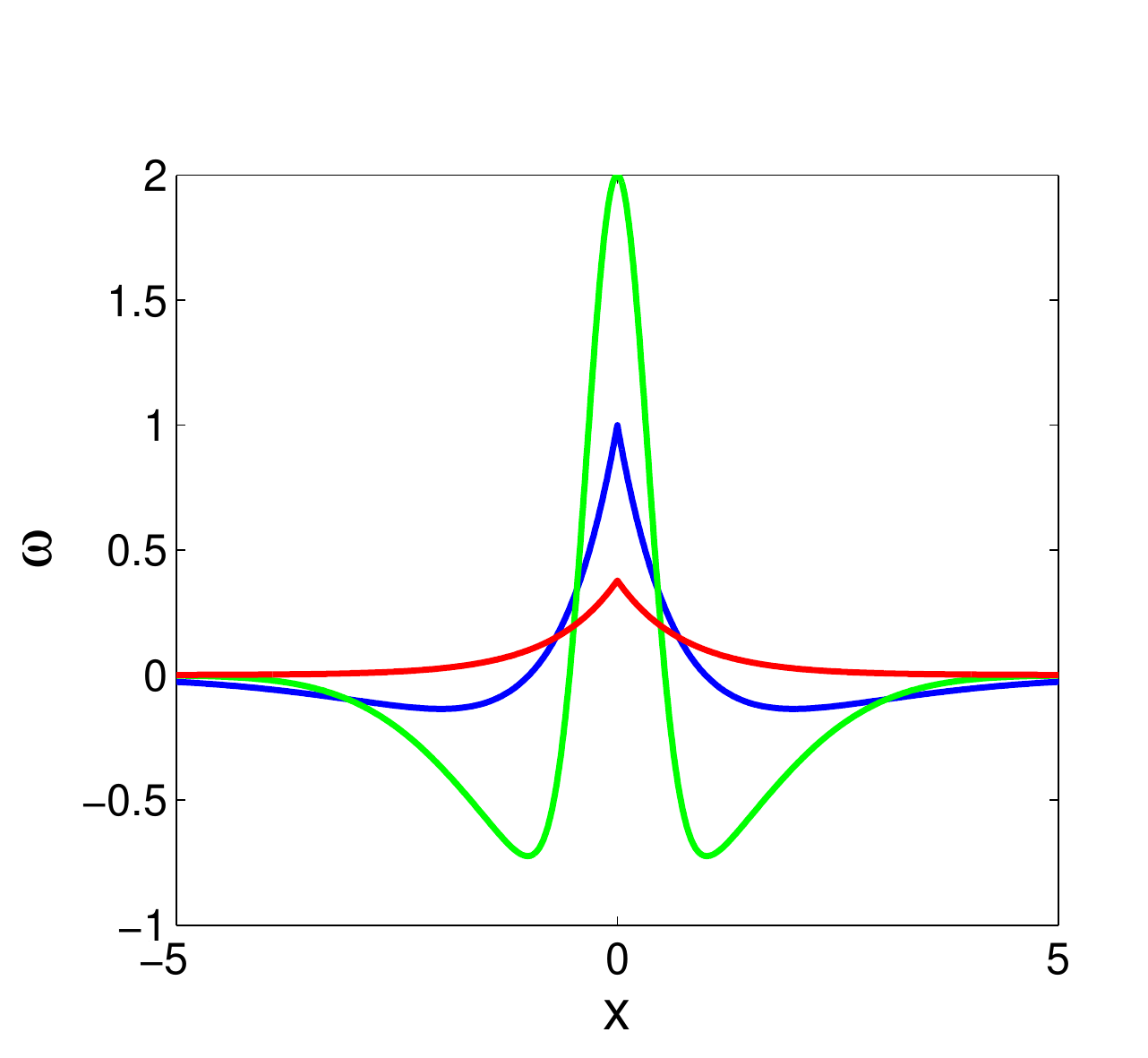}}
\caption{Functions $\omega$ given by \eqref{eq:omega:exp} with $k=1.3229$ (red curve), $\omega$ given by \eqref{eq:omega:wizard} with $k=1$ (blue curve) and $\omega$ given by \eqref{eq:omega:2Gaussians} with $K=3, k=2, M=1, m=0.5$ (green curve).  } \label{fig:omegas}
\end{figure}

When the Fourier transform of the connectivity function is real and rational, e.g., $\omega$ in \eqref{eq:omega:wizard}, the time independent version of \eqref{eq:1} can be converted to a higher order nonlinear differential equation which is in turn can be represented as a Hamiltonian system. Bumps correspond to homoclinic orbits of this system, \cite{Laing:2002,Laing:2003:a}. Despite there are some methods for studying existence of homoclinic orbits for higher order Hamiltonian systems, see e.g. \cite{Champneys:1998,Champneys:1999}, due to the specifics of the considered model these methods are not straightforwardly applicable here and the most results are only numerical, for details see \cite{Laing:2002}. For the case when $\omega$ does not admit a real rational Fourier transform, as for example in \eqref{eq:omega:2Gaussians}, the mentioned differential equations methods are not available and the other approaches must be used.

For $f=f_\beta,$ where $f_\beta \to f_\infty=\chi_{[h,\infty]}$ as $\beta\to \infty,$ and $\supp (f_\beta)\subset [h,\infty)$, the existence of bumps of a particular kind (so called $1$-bumps) and their continuous dependence on the parameter $\beta$ has been shown using the topological degree theory and collectively compactness and continuity of the Hammerstain operators in \cite{Oleynik:2013}. Later these results were extended to the case of $n$-dimensional, $n\in \N$ bounded spacial domain in \cite{Burlakov:2015}.

\section{Notation} \label{sec:Notation}
For readers convenience we give a list of functional spaces and specify other notations we use. Let $\Omega$ be a bounded or unbounded subset in $\R.$
\begin{itemize}
\item $L^p(\Omega),$ $1\leq p\leq \infty,$ is the space of all functions such that the $p$th power of the absolute value is Lebesgue integrable and the norm is given as
\begin{equation*}\label{eq:norm:Lp}
\|v\|_{L^p(\Omega)}:=\left(\int\limits_{\Omega}|v(x)|^p dx\right)^{1/p}, \quad 1\leq p<\infty,
\end{equation*}
and
\begin{equation*}\label{eq:norm:Linfty}
\|v\|_{L^\infty(\Omega)}:=\esup \limits_{x\in \Omega}|v(x)|.
\end{equation*}
When $\Omega=\R$ we use the common notation for the norm $\|v\|_\infty:=\|v\|_{L^\infty(\R)}.$
\\
\item $B(\Omega)$ is the linear space of all bounded functions.\\
\item $C(\Omega)$ is the linear space of all continuous (but not necessarily bounded) functions on $\Omega$.\\
\item $BC(\Omega)$ is the Banach space of all continuous bounded functions on $\Omega$ with the norm
\begin{equation*}\label{eq:norm:BC}
\|v\|_{C(\Omega)}:=\sup\limits_{x\in \Omega}|v(x)|.
\end{equation*}
When $\Omega=\R$ we often use $\|\cdot\|_\infty$ notation for the norm above.
\\
\item $C^{n}(\Omega)$, $n\in \N,$ is the linear space of all continuous (but not necessarily bounded) functions with continuous $k$th derivatives, $k=1,...,n,$ on $\Omega$.\\
\item $BC^{n}(\Omega)$, $n\in \N,$ is the Banach space of all continuous bounded functions with continuous and bounded $k$th derivatives, for $k=1,...,n,$ on $\Omega$ equipped with the norm
    \begin{equation*}\label{eq:norm:BCn}
\|v\|_{C^n(\Omega)}:=\sum \limits_{k=0}^{n}\sup\limits_{x \in \Omega}|v^{(k)}(x)|.
\end{equation*}
\\
\item $C^{0,1}(\Omega)$ is the space of all Lipschitz continuous functions on $\Omega$.\\
\item $C^{0,\alpha}(\Omega)$ is the space of all H\"{o}lder continuous functions on $\Omega$ with the exponent $0<\alpha<1$.\\
\end{itemize}
When $\Omega$ is a compact set then we prefer the notation $C^n(\Omega)$ over $BC^n(\Omega).$ Moreover, in this case we treat $C^{0,\alpha}(\Omega),$ $0<\alpha\leq 1$ as the Banach spaces equipped with the norm
\begin{equation*}\label{eq:norm}
\|v\|_{C^{0,\alpha}(\Omega)}:=\sup\limits_{x\in \Omega}|v(x)|+\sup \limits_{\begin{array}{c}
                                                            x,y\in \Omega\\
                                                            x\not=y \end{array}}
                                                            \dfrac{|v(x)-v(y)|}{|x-y|^\alpha}
\end{equation*}
When there is no confusion what $\Omega$ is we use the notation $\| \cdot\|_\alpha$ instead of $\|\cdot\|_{C^{0,\alpha}(\Omega)}.$

We denote $C_{even}(\Omega)$ the space of all even continuous function on $\Omega.$ The same notation applies for the other spaces, e.g., $L^p_{even}(\Omega),$ $1\leq p\leq \infty$ and etc.

 If $\Omega=\R$ we omit the set specification in the norm notation, i.e., we write $\|\cdot\|_{C^{n}}$ instead of $\|\cdot\|_{C^n(\R)}.$

We use boldface to denote vectors, e.g., ${\bf a},$ {\bf G}, and the textsf font for matrices, e.g., \textsf{S}, \textsf{I}.

\section{Framework and main results} \label{sec:Framework}

We study existence of solutions to the fixed point problem \eqref{eq:H}
under the following assumptions on $f_\beta$ and $\omega.$

\begin{assumption} \label{ass:f}
Let $h>0$ be fixed, and let $\{f_\beta\},$ $\beta\in(0,\infty]$ define a family of functions with the following properties:
\begin{itemize}
\item[(1)] $f_\beta:\R \to [0,1]$ is non-decreasing for any $\beta \in (0,\infty],$\\
\item[(2)]$\supp (f_\beta) \subset [h,\infty)$ for any $\beta \in (0,\infty],$\\
\item[(3)] $f_\infty=\chi_{(h,\infty)}$,i.e., the characteristic function of the half-line set $(h,\infty)$,\\
\item[(4)] $f_\beta(t)$ is continuous in $\beta \in (0,\infty)$ uniformly in $t$ on any bounded interval, \\
\item[(5)] $f_\beta \in C^{0,1}(\R)$ for $\beta<\infty,$ and for any $\xi>0$
$$
C_\beta(\xi):=\esup \limits_{t\in(h+\xi,\infty)}|f'_\beta(t)|\to 0 \mbox{ as } \beta \to \infty.
$$
\end{itemize}
\end{assumption}

The function in \eqref{eq:f:hill} with $p>1$ serves as an example of such a function.

We also notice that from Assumption \ref{ass:f}(5) $f_\beta$ has the following convergence property

\begin{equation}\label{eq:f:uniformly}
\sup \limits_{t\in \R\backslash(h-\xi,h+\xi)}|f_\beta(t) - f_\infty(t)| \to 0 \mbox{ as } \beta \to \infty.
\end{equation}

\begin{assumption} \label{ass:omega}
The function $\omega$ in \eqref{eq:H} satisfies the following conditions:
\begin{itemize}
\item[(1)] $\omega$ is symmetric, i.e., $\omega(-x)=\omega(x)$,\\
\item[(2)]$\omega$ is a Lipschitz function, i.e., $\omega \in C^{0,1}(\R),$ \\
\item[(3)] $\omega \in L^1(\R),$ and \\
\item[(4)] $\omega$ is bounded, i.e., $\omega \in L^\infty(\R).$
\end{itemize}
\end{assumption}

The functions in \eqref{eq:omega:exp}, \eqref{eq:omega:wizard}, and in \eqref{eq:omega:2Gaussians} clearly satisfy the assumption above.

Since the function $f_\beta$ is such that $f_\beta(u)=0$ for all $u<h,$ see Assumption \ref{ass:f}(2), all the solutions to \eqref{eq:H} can be divided into two categories: (i) localized solutions (so called bumps, see e.g. \cite{Amari:1977,Oleynik:2013}) and (ii) non-localized solutions (e.g., periodic, quasi-periodic).

Here we study the existence of solutions of the first type. We introduce only a few properties of the operator $\cH_\beta$ that is needed here. For more general description of $\cH_\beta$ we refer to \cite{Oleynik:2013} and \cite{Burlakov:2015}.

\subsection{Properties of $\cH_\beta$}\label{sec:Properties}

\begin{lemma}\label{lemma:1}
 Let $f_\beta$ and $\omega$ satisfy Assumption \ref{ass:f} and Assumption \ref{ass:omega}. Then for any real valued measurable function $u$ we have $\cH_\beta u \in BC(\R)$ for any $\beta \in (0,\infty]$. Moreover, if
$u(x)\leq h$ for all $x\in \R$ except a subset $X\subset\R$ of a finite measure,
then $\cH_\beta u \in L^1(\R) \cap BC^1(\R)$ and $(\cH_\beta u)(x) \to 0$ as $|x|\to \infty.$
\end{lemma}
\begin{proof}
For a general $u$ we define a set $X:=\{ x: \, u(x)\geq h\} \subseteq \R.$
We have the following estimate
$$
|(\cH_\beta u)(x)|=\left|\int_X \omega(x-y)f_\beta(u(y))dy\right| \leq
\int_X |\omega(x-y)|dy.
$$
Assumption \ref{ass:omega}(3) immediately yields $\|\cH_\beta u\|_\infty \leq \|\omega\|_{L^1}< \infty.$
To show continuity of $\cH_\beta$ let $v=\cH_\beta u$ and $x\to x_0.$  Then we obtain
$$
|v(x)-v(x_0)| \leq \int\limits_X |\omega(x-y)-\omega(x_0-y)|dy \to 0 \mbox{ as } x\to x_0,
$$
which follows from the continuity of translations in $L^1(\R),$ see, e.g., Proposition 2.5 in \cite{Stein:2005}.

Assume now that $X$ has a finite measure $\mu(X)<\infty.$  From Assumption \ref{ass:omega}(2) and Assumption \ref{ass:omega}(3) it follows that $\|\cH_\beta u\|_{L^1}\leq \|\omega\|_{L^1}\, \mu(X)<\infty.$
From Assumption \ref{ass:omega} (4) the derivative of $v=\cH_\beta u$ with respect to $x$ exists and is uniformly bounded, that is,
$$\|v'\|_\infty \leq \int_X |\omega'(x-y) f_\beta(u(y))|dy \leq \|\omega'\|_\infty \,\mu(X).$$

Next, we let $x\to x_0$ and obtain the estimate
$$|v'(x)-v'(x_0)|\leq \int \limits_X |\omega'(x-y)-\omega'(x_0-y)|dy.$$
 Assumption \ref{ass:omega}(2) implies that $\omega'\in L^\infty(\R)$ and thus $\omega' \in L^1(X).$ From the continuity of translations in $L^1(X)$ we deduce that  $|v'(x)-v'(x_0)| \to 0$ as $x\to x_0.$

 Hence we conclude that $\cH_\beta u \in BC^1(\R)$ and the
 the following estimate is valid
$$\|\cH_\beta u\|_{C^1}\leq (\|\omega\|_\infty+ \|\omega'\|_\infty) \mu(X)<\infty.$$

The property $v(x)\to 0$ as $|x|\to \infty$ follows from Assumption \ref{ass:omega}(2) and Assumption \ref{ass:omega}(3).
\end{proof}

From the lemma above any solution to \eqref{eq:H} is continuous and bounded.

\begin{lemma}\label{lemma:2}
Let $f_\beta$ and $\omega$ satisfy Assumption \ref{ass:f} and Assumption \ref{ass:omega}, respectively, and
let the operator $\cH_\beta:BC(\R) \to BC(\R)$ be defined as in \eqref{eq:H}.
Then the following statements are true.
(i) Any solution of \eqref{eq:H} is translation invariant, i.e., if $u(x)$ is a solution so is $u(x-c)$ for any $c\in \R.$ (ii) The operator $\cH_\beta$ preserves the symmetry, i.e., for any $u(x)=u(-x)$ we have $(\cH_\beta u)(-x)=(\cH_\beta u)(x).$
(iii) If for a fixed point $u(x)$ the corresponding $\supp(f_\beta(u(\cdot)))$ is a symmetric set, then $u(x)$ is even.
\end{lemma}

The proof is straightforward.

\begin{remark}
For Lemma \ref{lemma:1} and Lemma \ref{lemma:2} we did not use all the conditions of Assumption \ref{ass:f}. It suffices to assume that $u\mapsto f_\beta(u) \in [0,1]$ is measurable and  Assumption \ref{ass:f} (2)-(3) are satisfied.
\end{remark}

As mentioned before we intend to investigate the existence of localized solutions to \eqref{eq:H}. In the next section we describe the class of functions we are interested in.

\subsection{Bumps and regular bumps}\label{sec:Bumps}
\begin{definition}
Let $h\in \R$ and $\{b_i\}_{i=1}^{2N} \subset \R$ be an increasing sequence of $2N$ points.
The function $u \in C(\R)$ is called a $(h;\, {b_1,\,b_2,...,\,b_{2N}})$-bump if the following conditions are satisfied:\\
\begin{itemize}
\item[(i)] $\{b_i\}_{i=1}^{2N}$ are the only roots to $u(x)=h,$ \\
\item[(ii)]  there exists $\gamma>0$ and $A>0$ such that $u(x)<h-\gamma$ for all $|x|>A.$\\
\end{itemize}

We call the $(h;\, {b_1,\,b_2,\ldots,\,b_{2N}})$ - bump regular if in addition $u\in C^1(\R)$ and $u'(b_i) \not=0$ for all $i=1,\ldots,2N.$

When $h$ is assumed to be given and there is no need to specify the roots $\{b_i\}_{i=1}^{2N}$, we often refer to $u$ as a (regular) N-bump or, simply a (regular) bump.
\end{definition}
We illustrate the definition with Fig. \ref{fig:Bumps}.

\begin{figure}[h]
\scalebox{0.7}{\includegraphics{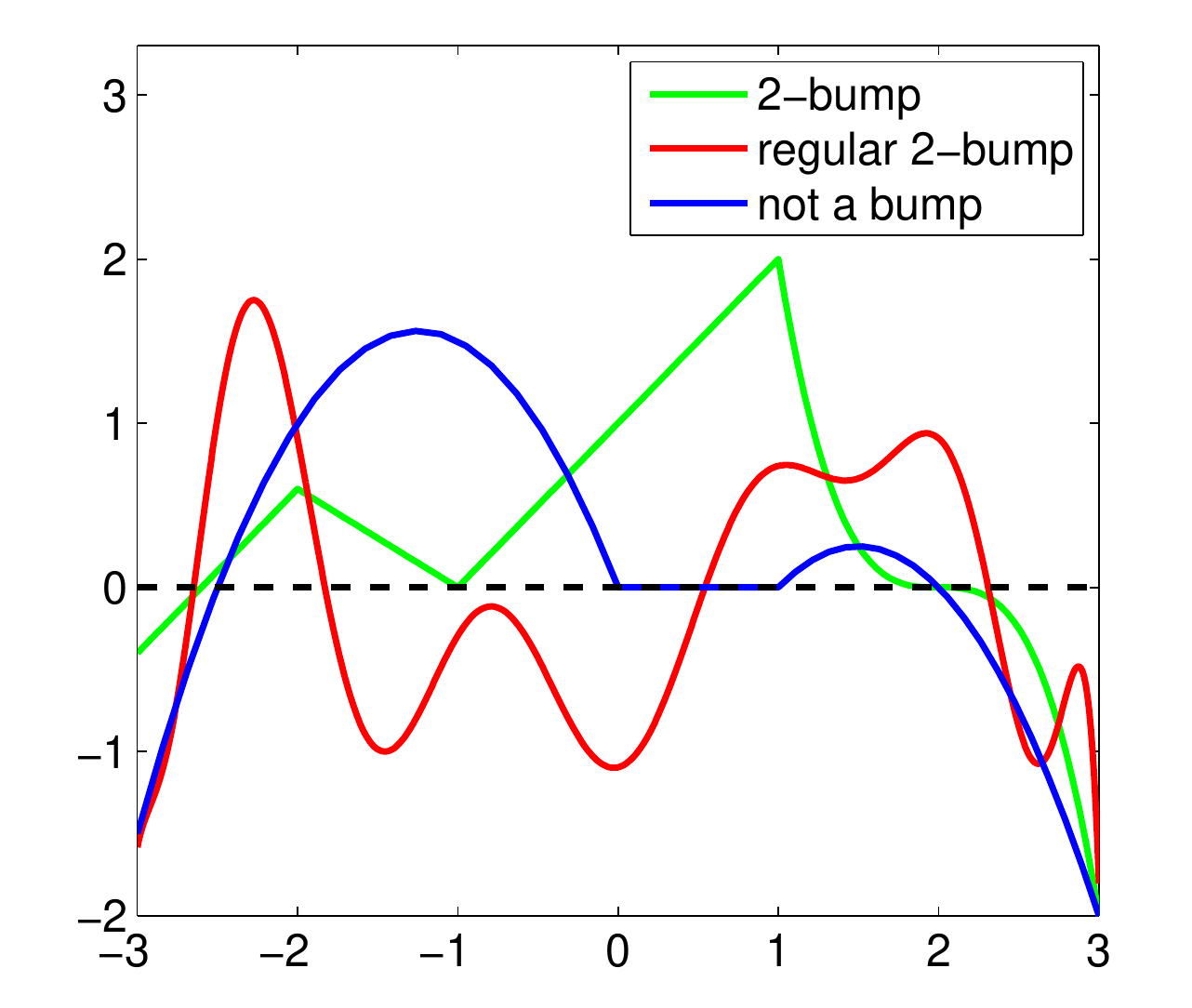}}
\caption{Examples of a bump, regular bump and not a bump, when $h=0$ and assuming that the functions are smooth enough and do not increase outside the interval $[-3,3].$} \label{fig:Bumps}
\end{figure}

 Regular bumps are stable under small perturbations in $C^1(\R).$ Indeed, let $u$ be a regular $(h; \,b_1,\,b_2,\ldots,\,b_{2N})$ - bump, and for some $l<b_1$ and $L>b_{2N}$ define the set
\begin{equation} \label{eq:K}
K_\epsilon(u;l,L):=\{v\in C^1(\R): \, \sup \limits_{x\in \R} |u(x)-v(x)|+\max\limits_{x\in [l,L]} |u'(x)-v'(x)|<\epsilon\}.
\end{equation}

We formulate the following lemma.
\begin{lemma} \label{lemma:3}
Let $u$ and $K_\epsilon(u;l,L)$ be given as above. Then there exists $\epsilon \in \left(0, 1/2 \min_{i} |u'(b_i)| \right)$  such that for any $v\in K_\epsilon(u;l,L),$ $v$ is a regular
$(h;\, {b_1(\epsilon),\,b_2(\epsilon),\ldots,\,b_{2N}(\epsilon)})$ - bump where $b_i(\epsilon) \to b_i$ as $\epsilon \to 0.$
\end{lemma}

The proof is rather straightforward and can be found in \cite{Oleynik:2013}.

\begin{corollary}\label{cor:lemma:3}
Observe that there are values $b_i^{\pm}(\epsilon):\,b_i^{-}(\epsilon)<b_i<b_i^{+}(\epsilon)$ where $\bigcap \limits_{i=1}^{2N}(b_i^{-}(\epsilon),b_i^{+}(\epsilon))=\varnothing$ associated with $K_\epsilon(u;l,L)$ such that $b_i^{\pm}(\epsilon)\to b_i$ as $\epsilon\to 0,$ and for any $v\in K_\epsilon(u;l,L)$ we have the following estimates
\begin{equation*}
\left\{
\begin{array}{ll}
|v(x)-h|>\epsilon, & x\in \R \setminus \bigcup \limits_{i=1}^{2N} (b_i^{-}(\epsilon),b_i^{+}(\epsilon))\\
|v'(x)|>\epsilon, &  x\in \bigcup \limits_{i=1}^{2N} (b_i^{-}(\epsilon),b_i^{+}(\epsilon))\\
\end{array}
\right.
\end{equation*}
\end{corollary}

\begin{definition}
A (regular) bump that is a solution to \eqref{eq:H} we call a (regular) bump solution.
\end{definition}

Lemma \ref{lemma:1} implies that any bump solution $u(x)$ of \eqref{eq:H} is in $BC^1(\R)$ and $|u(x)|\to 0$ as $|x|\to \infty.$ Thus, the threshold $h$ cannot be chosen negative.

\subsection{Bump solutions to \eqref{eq:H} with $\beta=\infty$}\label{sec:Limiting:problem}

In \cite{Amari:1977} Amari studied the equation \eqref{eq:H} under the assumption that $f_\beta=f_\infty=\chi_{[h,\infty)}$ with some $h>0.$ In this case one can find analytic expressions for the bump solutions. Let $\omega$ satisfy Assumption \ref{ass:omega} and suppose that \eqref{eq:H} with $\beta=\infty$ has a $(h;\,b_1,\ldots,b_{2N})$-bump solution, say $u_\infty.$

 Then it immediately follows from \eqref{eq:H} that
\begin{equation*}
u_\infty(x)=\sum\limits_{k=1}^{N}\, \int \limits_{b_{2k-1}}^{b_{2k}}\omega(x-y)dy
\end{equation*}
or, rewriting the equation above in terms of the anti-derivative
\begin{equation*}
W(x)=\int\limits_0^x \omega(y)dy,
\end{equation*}

\begin{equation}\label{eq:u:inf}
u_\infty(x)=\sum\limits_{k=1}^{N} \Big(W(x-b_{2k-1})-W(x-b_{2k})\Big).
\end{equation}

 Next, we assume that $u_\infty$ is symmetric and, thus, $b_i=-a_{N-i+1},$ and $b_{N+i}=a_{i},$ for $i=1,\ldots,N.$ In this new notation, $u_\infty$ is a $(h;\,-a_N,\ldots,-a_1,a_1,\ldots,a_N)$-bump and we can rewrite \eqref{eq:u:inf} as

\begin{equation}\label{eq:u:inf:new}
u_\infty(x)=\sum\limits_{k=1}^{N} (-1)^{N-k+1}\Big(W(x-a_{k})-W(x+a_{k})\Big).
\end{equation}

The vector ${\bf a}=(a_1,\ldots,a_N)^T$ with $a_i<a_{i+1},$ $i=1,\ldots,N-1,$ must be a solution of the system of $N$ nonlinear equations
\begin{equation} \label{eq:roots}
u_\infty(a_i)=\sum\limits_{k=1}^{N} (-1)^{N-k+1}\Big(W(a_i-a_{k})-W(a_i+a_{k})\Big)=h.
\end{equation}

Once ${\bf a}$ is found one can construct $u_\infty$ using the formula \eqref{eq:u:inf:new} and then  verify that the obtained function is indeed a bump, and thus, a bump solution to \eqref{eq:H}. By Lemma \ref{lemma:2}(iii) the function $u$ is automatically even.

Alternatively, when $\omega$ has a rational real Fourier transform one can obtain an analytical expression for bump solutions by solving the corresponding piecewise linear ordinary differential equation, see e.g. \cite{Laing:2002,Laing:2003:a,Krisner:2004}.
 We do not focus on this problem here but refer the reader to \cite{Amari:1977,Laing:2002,Murdock:2006} for more details.

  Further on we assume that the bump solution exists, it is symmetric, regular, and impose one extra assumption on the intersection points $a_i$ which role will be more clear later.

\begin{assumption}\label{ass:u:infty}
Let $h>0$ be fixed and $u_\infty$ given by \eqref{eq:u:inf:new} be a $(h;-a_N,\ldots,-a_1,a_1,\ldots,a_N)$-bump solution to \eqref{eq:H} with $f_\beta=f_\infty=\chi_{[h,\infty)}.$
\begin{itemize}
\item[(1)] The solution $u_\infty$ is a regular bump, that is,
\begin{equation}\label{eq:regularity}
|u'_\infty(a_i)|= \sum \limits_{k=1}^N (-1)^{k+i}(\omega(a_i-a_k)-\omega(a_i+a_k))>0,\quad i=1,...,N.
\end{equation}
\item[(2)] From \eqref{eq:roots} ${\bf a}$ is a solution to ${\bf G}({\bf a})={\bf 0}$ where
${\bf G}=(g_1,...,g_N)^T$ with $g_i=(-1)^{i+1}(u_\infty(a_i)-h).$
We assume that the Jacobian matrix of ${\bf G}$ derived at ${\bf a}$, $J({\bf a})=\Big(\partial g_i/\partial a_j\Big)_{ij}({\bf a})$ with
\begin{equation*}
\begin{array}{l}
\dfrac{\partial g_i}{\partial a_i}({\bf a})=|u'_\infty(a_i)|-\omega(0)-\omega(2a_i)\\
\\
\dfrac{\partial g_i}{\partial a_j}({\bf a})=(-1)^{i+j+1}(\omega(a_i-a_j)+\omega(a_i+a_j)),
\end{array}
\quad
\begin{array}{c}
i,j=1,\ldots,N,\\
 i\not=j,
\end{array}
\end{equation*}
has an inverse, that is, $\det(J({\bf a})) \not=0.$
\end{itemize}
\end{assumption}
In particular, Assumption \ref{ass:u:infty}(ii) guarantees that ${\bf a}$ is the isolated solution of ${\bf G}({\bf a})={\bf 0}.$

\subsection{Existence and approximation of bump solutions for large $\beta<\infty$ } \label{sec:Main:results}

We formulate our main results.

\begin{theorem}\label{th:1}
Let $h>0$ be fixed and $f_\beta$ and $\omega$ satisfy Assumptions \ref{ass:f} and \ref{ass:omega}.
Moreover, we assume that $\omega$ is such that for $\beta=\infty$ there exists a symmetric $(h; -a_N,\ldots,-a_1,a_1,\ldots,a_{N})$-bump solution $u_\infty$ of \eqref{eq:H} and Assumption \ref{ass:u:infty} is satisfied.
Then we have the following result.
\begin{itemize}
\item There is an $\epsilon>0$ such that for sufficiently large $\beta>0$ the operator
$\cH_\beta:K_\epsilon(u_\infty;-d,d) \to BC^1(\R)$ for any $d>a_N$ has a symmetric fixed point $u_\beta$ which is a regular bump. Moreover $\|u_\beta-u_\infty\|_{C^1}\to 0$ as $\beta \to \infty.$ \\

\item The bump solution $u_\beta \in K_\epsilon(u_\infty;-d,d)$ can be iteratively constructed and the sequence of successive approximations $\{u_{n}\}_{n\in\N\cup\{0\}}$
converges to the solution $u_\beta$ in $BC^1(\R)$-norm.  The sequence $u_{n},$ $n=0,1,\ldots,$ is defined by

\begin{subequations}\label{eq:u-appx}
\begin{align}
u_{n+1}=\int\limits_{-d}^{d}\omega(x-y) f_\beta(U_{n+1}(y))dy, \quad x \in \R \label{eq:u-appx:a}
\\
U_{n+1}(x)=\left(H_\beta(U_n)\right)(x)-{\bf s}^{T}(x) (\textsf{S}-\textsf{I})^{-1} {\bf p}_{n}, \quad x\in [-d,d]
\label{eq:u-appx:b}
\end{align}
\end{subequations}
with $U_0$ being the restriction of $u_\infty$ on $[-d,d],$

$$
\left( H_\beta U \right)(x):=\int\limits_{-d}^d \omega(x-y)f_\beta(U(y))dy \in C^1([-d,d]),
$$

$$
{\bf s}(x)=\left(\frac{\omega(x-a_1)+\omega(x+a_1)}{|u'_\infty(a_1)|},\ldots,
\frac{\omega(x-a_N)+\omega(x+a_N)}{|u'_\infty(a_N)|}\right)^T,
$$

and $\textsf{S}$ is an $N\times N$ matrix with the elements

$$
S_{ij}=\frac{\omega(a_i-a_j)+\omega(a_i+a_j)}{|u'_\infty(a_j)|}, \quad i,j=1,\ldots,N,
$$

and $\textsf{I}$ is the $N\times N$ identity matrix. The vector ${\bf p}_n$ is defined as
$$
{\bf p}_{n}:=(p_n(a_1),...,p_n(a_N))^T, \quad p_n(a_i)=\left(H_\beta U_n-U_n\right)(a_i),\quad i=1,\ldots,N.
$$
\end{itemize}
\end{theorem}

\section{Proof of Theorem \ref{th:1}}\label{sec:Proof}

Let $h>0$ be fixed and let $f_\beta$ and $\omega$ satisfy Assumption \ref{ass:f} and Assumptions \ref{ass:omega}, respectively. Moreover, we assume that $\omega$ is such that for $\beta=\infty$ there exists a symmetric regular \\
$(h; -a_N,\ldots,-a_1,a_1,\ldots,a_{N})$ - bump solution $u_\infty$ and Assumption \ref{ass:u:infty} is satisfied.

Let $d=a_N+\delta$ for some $\delta>0$ and define a set of even functions
\begin{equation*}
K_\epsilon(u_\infty):=K_\varepsilon(u_\infty; -d,d)\cap C^1_{even}(\R)
\end{equation*}
where $K_\varepsilon(u_\infty; -d,d)$ is given as in \eqref{eq:K}.
We assume that $\epsilon>0$ is small enough such that $K_\epsilon(u_\infty)$ contains only regular symmetric $N$-bump solutions, see Lemma \ref{lemma:3}.

If $u\in K_\epsilon(u_\infty)$ is a solution to \eqref{eq:H} then due to Corollary \ref{cor:lemma:3} and the choice of $d,$ it is given as
$$u(x)=\int_{-d}^{d} \omega(x-y)f_\beta(u(y))dy.$$

Now let $T_1: BC^1(\R) \to B([-d,d])$ be the restriction operator given as $T_1u=u|_{[-d,d]},$
 and $T_2$ be the reconstruction operator
 \begin{equation}\label{eq:T2}
 (T_2 U)(x)=\int\limits_{-d}^d \omega(x-y)f_\beta(U(y)-h)dy,
 \end{equation}
 where by Lemma \ref{lemma:1} we have $T_2: B([-d,d]) \to BC^1(\R).$

From now on we use capital letters to denote the restriction of $u\in K_\epsilon(u_\infty),$ that is,
$U=T_1u$  and, in particular, $U_\infty=T_1u_\infty.$
Notice that $T_1 K_\epsilon(u_\infty)=B_\epsilon(U_\infty)$ is the $\epsilon$-ball in $C^1_{even}([-d,d]),$ that is,
\begin{equation*}
  B_\epsilon(U_\infty):=\{U :\, \|U-U_\infty\|_{C^1([-d,d])}<\epsilon\} \cap C^1_{even}(\R).
\end{equation*}

It is obvious that Lemma \ref{lemma:3} and Corollary \ref{cor:lemma:3} can be directly reformulated for $B_\epsilon(U_\infty)$ as $a_i\in [0,d),$ $i=1,\ldots,N$.
We formulate this as a remark.

\begin{remark}\label{rem:3}
From Corollary \ref{cor:lemma:3} there are $a_i^{\pm}(\epsilon),$ $i=1,\ldots,N,$ such that $a_i^{\pm}(\epsilon) \to 0$ as $\epsilon \to 0$ and for any $U\in B_\epsilon(U_\infty)$
\begin{subequations}\label{ineq:bumps}
\begin{align}
|U(x)-h|>\epsilon, & \quad x\in \R \setminus \bigcup \limits_{i=1}^{N} (\pm a_i^{\mp}(\epsilon),\pm a_i^{\pm}(\epsilon)),\label{ineq:bumps:a}
\\
|U'(x)|>\epsilon, &  \quad x\in \bigcup \limits_{i=1}^{N} (\pm a_i^{\mp}(\epsilon),\pm a_i^{\pm}(\epsilon)).
\label{ineq:bumps:b}
\end{align}
\end{subequations}
\end{remark}

Hence, if $u\in K_\epsilon(u_\infty)$ is the solution to
\eqref{eq:H} then $u=T_2(T_1 u)$ and $U=T_1u \in B_\epsilon(U_\infty)$ is a solution to the fixed point problem
\begin{equation} \label{eq:H2}
U=H_\beta U, \quad (H_\beta U)(x):=\int\limits_{-d}^{d} \omega(x-y)f_\beta(U(y))dy.
\end{equation}

On the other hand, if $U\in B_\epsilon(U_\infty)$ is the solution to \eqref{eq:H2}
there is no guarantee that any $T_1$ - preimage of $U$ is a fixed point of $\cH_\beta$ in \eqref{eq:H}.
However, if it is, then it must be given as $u=T_2U.$

With the next proposition we claim there exist sufficiently small $\epsilon>0$ and sufficiently large $\beta>0$ such that for any solution $U\in B_\epsilon(U_\infty)$ of \eqref{eq:H2} the corresponding $u=T_2U$ is a $N$-bump solution to \eqref{eq:H}. We need an auxiliary lemma.

\begin{lemma}\label{lemma:4}
The Nemytskii operator $N(\beta;U)=f_\beta(U):(0,\infty) \times L^\infty([-d,d])\to L^p([-d,d]),$ $1\leq p \leq \infty$ is jointly continuous in $(\beta_0, U_0)$ for any $\beta_0 \in (0,\infty)$ and $U_0 \in L^\infty([-d,d]).$  Moreover, $N(\beta;U)$ is jointly continuous in $(\infty, U_0)$ for any $U_0 \in B_\epsilon(U_\infty)$ as a map from $(0,\infty] \times B_\epsilon(U_\infty)$ to $L_{even}^p([-d,d])$ for $1\leq p <\infty.$
\end{lemma}
\begin{proof}

Note that $f_\beta,$ $\beta \in (0,\infty]$ is uniformly bounded and thus, $N(\beta,U)$ is in $L^p([-d,d])$
 and, if $U$ is even, in $L^p_{even}([-d,d])$, $1\leq p\leq \infty.$

Let $\beta, \beta_0\in (0,\infty]$ and $U, U_0 \in L^\infty([-d,d])$ then we have the estimate

\begin{equation*}
\begin{split}
\|N(\beta,U)-N(\beta_0,U_0)\|^p_{L^p([-d,d])}=&\int\limits_{-d}^d \left|f_{\beta}(U(x))-f_\infty(U_0(x))\right|^p dx \leq \\
&\Sigma_1(U)+\Sigma_2(\beta_0)
\end{split}
\end{equation*}
where
\begin{equation*}
\Sigma_1(U):= \int\limits_{-d}^d \left|f_{\beta}(U(x))-f_{\beta_0}(U(x))\right|^p dx,
\end{equation*}
and
\begin{equation*}
\Sigma_2(\beta_0):= \int \limits_{-d}^d \left|f_{\beta_0}(U(x))-f_{\beta_0}(U_0(x))\right|^p dx.
\end{equation*}

For  $\beta\to \beta_0,$ $\beta,\beta_0 \in (0,\infty),$ $\Sigma_1(U) \to 0$ uniformly in $U\in L^\infty([-d,d])$ by Assumption \ref{ass:f}(4). From Assumption \ref{ass:f}(5) we obtain
\begin{equation*}
\Sigma_2(\beta_0)\leq 2d \|f'_{\beta_0}\|_\infty \|U-U_0\|_{L^\infty([-d,d])} \to 0 \mbox{ as }
\|U-U_0\|_{L^\infty([-d,d])} \to 0.
\end{equation*}
Thus, $(0,\infty) \times L^{\infty}([-d,d]) \mapsto N(\beta,U)$ is jointly continuous in $(\beta_0,U_0).$

Now we let $\beta\to \infty$ and show that $\Sigma_1(U)$  converges to zero uniformly for all $U\in B_\epsilon(U_\infty).$ We introduce $D:=\bigcup_{i=1}^N (a_i^-(\epsilon),a_i^-(\epsilon))$ with $a^{\pm}_i(\epsilon)$ given as in \eqref{ineq:bumps}. Then we have
\begin{equation} \label{eq:Sigma:1}
\begin{split}
\Sigma_1(U)=& \int\limits_{[-d,d]/D} \left|f_{\beta}(U(x))-f_{\beta_0}(U(x))\right|^p dx +\\
&\int\limits_D \left|f_{\beta}(U(x))-f_{\beta_0}(U(x))\right|^p dx.
\end{split}
\end{equation}

Assumption \ref{ass:f}, or more precisely \eqref{eq:f:uniformly},  yields
the uniform in $U$ convergence of the first integral in \eqref{eq:Sigma:1} as $\beta \to \infty.$

To estimate the second integral in \eqref{eq:Sigma:1} note that $|U'(x)|>\epsilon$ on $(a_i^{-}(\epsilon),a_{i}^{+}(\epsilon)),$ $i=1,\ldots,N,$ see Remark \ref{rem:3}, \eqref{ineq:bumps:b}.  Thus, there exists an inverse of $U(x),$ $k_i(t):(U(a_i^{-}(\epsilon)),U(a_i^{+}(\epsilon)))\to (a_i^{-}(\epsilon),a_i^{+}(\epsilon))$ with $|k_i'(t)|=1/|U'(x)|\leq 1/\epsilon.$

Define $M:=\max_{x\in [-d,d]}U_\infty(x)+\epsilon$ and $m:=\min_{x\in [-d,d]}U_\infty(x)-\epsilon$. We  have
\begin{equation*}
\begin{split}
 \int\limits_{a_i^-(\epsilon)}^{a_i^+(\epsilon)} \left|f_{\beta}(U(x))-f_{\infty}(U(x))\right|^p dx=
 &\int\limits_{U(a_i^{-}(\epsilon))}^{U(a_i^{+}(\epsilon))}\left|f_{\beta}(t)-f_{\infty}(t)\right| |k'_i(t)|dt \leq \\
 &\epsilon^{-p} \int\limits_m^M |f_{\beta}(t)-f_\infty(t)|^p dt.
 \end{split}
\end{equation*}
 Hence,
$$\int\limits_D \left|f_{\beta}(U_n(x))-f_{\beta_0}(U_0(x))\right|^p dx \leq N \epsilon^{-p} \int\limits_m^M |f_{\beta}(t)-f_\infty(t)|^p dt \to 0, \quad \beta \to \infty$$
by the Lebesgue dominant convergence theorem. Thus $\Sigma_1(U) \to 0$ as $\beta\to \infty.$

Next, we consider $\Sigma_2(\infty)$. Let $U=U_n, U_0 \in B_\epsilon(U_\infty)$ and  $\|U_n-U_0\|_{C^1([-d,d])}\to 0$ as $n\to \infty.$ In order to avoid introducing new notation we assume $U_0=U_\infty.$ For $U_0\not=U_\infty$ the same analysis applies. Note that as $n\to \infty,$ $U_n \in B_{1/n}(U_\infty).$ Define $D_n:=\bigcup_{i=1}^N (a_i^-(1/n),a_i^+(1/n))$ with $a_i^{\pm}(\epsilon),$ $\epsilon=1/n,$ see Remark \ref{rem:3}. Observe that as $n\to \infty$  the sequences  $a^{\pm}_i(1/n)\to a_i$  and $\mu(D_n)\to 0.$

From \eqref{ineq:bumps:a} $f_\infty(U_n(x))=f_\infty(U_\infty(x))$ for all $x\in [-d,d]\setminus D_n$ and therefore
$$\Sigma_2(\infty)= \int_{D_n} |f_\infty(U_n(x))-f_\infty(U_0(x))|^p dx \leq 2^{p+1} \mu(D_n) \to 0 \mbox{ as } n\to 0.$$

Convergence properties of $\Sigma_1(U)$ and $\Sigma_2(\infty)$ result in the joint continuity of $(0,\infty]\times B_\epsilon(U_\infty) \mapsto N(\beta, U)$ at $(\infty,U_0).$
\end{proof}

\begin{proposition}\label{prop:1}
Let $f_\beta$ and $\omega$ satisfy Assumptions \ref{ass:f} and Assumptions \ref{ass:omega}, respectively.
Moreover, there exists $(h;-a_N,\ldots,-a_1,a_1,\ldots,a_N)$-bump solution $u_\infty$ of \eqref{eq:H} for $\beta=\infty$ that satisfies Assumption \ref{ass:u:infty}, and $U_\infty=T_1u_\infty$ being its restriction on $[-d,d].$
Then there exists $\epsilon>0$  such that starting from sufficiently large $\beta,$ if $U_\beta \in B_\epsilon(U_\infty)$ is a solution to \eqref{eq:H2} then
$u_\beta:=T_2 U_\beta$
is the solution to \eqref{eq:H}.
\end{proposition}

\begin{proof}
It is sufficient to show that $u_\beta \in K_\rho(u_\infty)$ where $\rho>0$ is small enough so
that $u_\beta$ is a regular $N$-bump, see Lemma \ref{lemma:3}. We derive the estimate

\begin{equation*}
\begin{split}
\|u_\beta-u_\infty\|_\infty &\leq \int\limits_{-d}^d |\omega(x-y)
(f_\beta(u_\beta(y))-f_\infty(u_\infty(y))) dy| \leq \\
& \|\omega\|_{\infty} \left(\int\limits_{-d}^d |f_\beta(u_\beta(y))-f_\beta(u_\infty(y))| dy \right.+
\left.\int\limits_{-d}^d |f_\beta(u_\infty(y))-f_\infty(u_\infty(y))|dy \right)\leq\\
&\|\omega\|_{\infty} \Big( 2d \|f'_\beta\|_\infty \|U_\beta-U_\infty\|_\infty+\|N(\beta,U_\infty)-N(\infty,U_\infty)\|_{L_1([-d,d])} \Big).
\end{split}
\end{equation*}
From Lemma \ref{lemma:2} $N(\beta,U_\infty)$ is continuous in $\beta$ and, thus, $\|N(\beta,U_\infty)-N(\infty,U_\infty)\|_{L_1([-d,d])}\leq \epsilon$ for sufficiently large $\beta>0.$

Thus, we have
\begin{equation*}
\|u_\beta-u_\infty\|_\infty +\|u'_\beta-u'_\infty\|_{C([-d,d])}\leq
\epsilon \Big(1+\|\omega\|_{\infty}  + 2d \|f'_\beta\|_\infty \|\omega\|_\infty \Big).
\end{equation*}

Assigning $\epsilon< \rho/\Big(1+\|\omega\|_{\infty} +2d \|f'_\beta\|_\infty \|\omega\|_\infty \Big)$ we secure that $u_\beta\in K_\rho(u_\infty).$ This completes the proof.
\end{proof}

In the view of the proposition above we can study existence of solutions to  \eqref{eq:H} with $\cH_\beta:K_\rho(u_\infty)\to C^1_{even}(\R)$ by studying existence of the solutions to \eqref{eq:H2} on $B_\varepsilon(U_\infty).$

We make use the following classical result.

 \begin{theorem}[Implicit Function Theorem, e.g., Section 4.7 in \cite{Zeidler:I}] \label{th:IFT}
 Let $\cV$, $\cU$, and $\cW$ be Banach spaces, $(v_0,u_0)\in \cV \times \cU,$ and $\Omega \subset \cV \times \cU$ be a neighbourhood of $(v_0,u_0).$
 Let the operator $P:\Omega \to \cW$ satisfy the following properties
 \begin{itemize}
 \item[(i)] $P(v_0,u_0)=0$, \\
 \item[(ii)]$P$ is continuous at $(v_0,u_0),$ \\
 \item[(iii)] there exist $ \Omega \mapsto P'_u[v,u]$ such that it is continuous in $(v_0,u_0)$, i.e.,
  \begin{equation*}
 \lim \limits_{(v,u)\to (v_0,u_0)} \|P'[v,u]-P'[v_0,u_0]\|_{\cW}=0
 \end{equation*}\\
 \item[(iv)] the operator $P'_u[v_0,u_0]:\cU \to \cW$ is a bounded linear operator with the bounded inverse
 $$ \Gamma=(P'_u[v_0,u_0])^{-1}:\cW \to \cU.$$
 \end{itemize}

 Then the following are true:
 \begin{itemize}
 \item
There exist an operator $F:O\to \cU$, where $O\subset \cV$ is some neighbourhood of $v_0,$ with the following properties
 \begin{itemize}
 \item[(a)] $P(v,Fv)=0$ for all $v \in O$ \\
 \item[(b)] $Fv_0=u_0$\\
 \item[(c)] $F$ is continuous in $v_0$.
 \end{itemize}
Moreover, the operator $F$ is uniquely defined, i.e., if there exists $F_1$ that satisfies (a)-(c) then there is $\varepsilon>0$ such that  $F_1v=F v$ for all $\|v-v_0\|_{\cV}<\epsilon.$\\
\\
\item  The sequence of successive approximations $\{ F_n\}$ defined by $F_0v \equiv  u_0$ and
\begin{equation} \label{eq:Iteration}
F_{n+1}v=F_nv- (P'_u[v_0,u_0])^{-1} \circ P(v, F_{n}v)
\end{equation}
converges to the solution $Fv$ as $n\to \infty$ for all $v \in O.$
\end{itemize}
 \end{theorem}

Define the operator $P(\beta,U): (0,\infty]\to C^{0,\alpha}_{even}([-d,d])$ for some $0<\alpha<1$ as
\begin{equation} \label{eq:op:P}
P(\beta,U)=-U+H_\beta(U).
\end{equation}
Using the notations in Theorem \ref{th:IFT} we have $\cV=\R,$ $\cU=C^1_{even}([-d,d]),$ $\cW=C^{0,\alpha}_{even}([-d,d])$ and $\Omega=(0,\infty]\times B_\varepsilon(U_\infty).$
Though it follows from Lemma \ref{lemma:1} and Lemma \ref{lemma:2} that $(0,\infty]\times B_\epsilon(U_\infty) \mapsto P(\beta,U) \in C^1_{even}([-d,d])$ we were not able to prove the condition (iii) of Theorem \ref{th:IFT} for $C^1_{even}([-d,d])$ but $C^{0,\alpha}_{even}([-d,d]).$ We will comment on it later. \\

We outline the idea of the proof of Theorem \ref{th:1}.
\begin{itemize}
\item[Step 1.] We verify the conditions (i)-(iiv) of Theorem \ref{th:IFT} for the operator $P$ in \eqref{eq:op:P}.\\
\item[Step 2.] We prove the first part of Theorem \ref{th:1}. First, we apply the first part of Theorem \ref{th:IFT} to the operator $P$ to obtain existence and uniqueness of the fixed point $U_\beta \in B_\epsilon (U_\infty)$ of the operator $H_\beta$ for large $\beta>0$ and $\|U_\beta-U_\infty\|_{C^1([-d,d])}\to 0$ as $\beta\to \infty.$ Next, we define $u_\beta=T_2U_\beta$ with $T_2$ given in \eqref{eq:T2} and use Proposition \ref{prop:1} to obtain the results for $u_\beta.$\\
\item[Step 3.] We prove the second part of Theorem \ref{th:1}. Firstly, we validate \eqref{eq:u-appx:b} using Theorem \eqref{th:IFT}. Secondly, we show that $u_n=T_2U_n$ converges to $u_\beta$.\\
\end{itemize}

\subsection*{Step 1}
The condition (i) of Theorem \ref{th:IFT} follows directly for $\beta=\infty$ and $U=U_\infty,$ that is,  $P(\infty,U_\infty)=0.$ The condition (ii) follows from Lemma \ref{lemma:4}. Indeed,
\begin{equation*}
\|H_\beta(U)-H_\infty(U_\infty)\|_{C^1([-d,d])}\leq
\left(\|\omega\|_\infty+\|\omega'\|_\infty\right)\|N(\beta;U)-N(\infty;U_\infty)\|_{L^1([-d,d])}\to 0
\end{equation*}
as $\|U-U_\infty\|_{C^1([-d,d])}\to 0$ and $\beta\to \infty$ due to Lemma \ref{lemma:4} which implies the continuity of the operator $P(\beta;U)$ at $(\infty, U_\infty).$

Next we show Fr\'{e}chet differentiability of the operator $H_\beta$ for $\beta<\infty.$
\begin{lemma}\label{lemma:5:1}
The operator $H_\beta:B_\epsilon(U_\infty)\subset C_{even}^1([-d,d])\to C_{even}^1([-d,d]),$ $\beta<\infty,$ given in \eqref{eq:H2} is Fr\'{e}chet differentiable at $U\in B_\epsilon(U_\infty)$ with the derivative
\begin{equation} \label{eq:S}
H'_\beta(U)=S[\beta,U], \quad S[\beta,U]v=\int\limits_{-d}^d \omega(x-y)f'_\beta(U(y))v(y)dy.
\end{equation}
\end{lemma}

\begin{proof}
Computing the G\^{a}teaux derivative of $H_\beta$ at $U\in B_\epsilon(U_\infty)$ we obtain $dH_\beta(U;v)=S[\beta,U]v,$ with $S[\beta,U]$ given in \eqref{eq:S}. The operator $H_\beta$ is Fr\'{e}chet differentiable
if
$$
\lim \limits_{t\to 0} \left\| \dfrac{H_\beta(U+tv)-H_\beta(U)}{t} -S[\beta,U]v \right\|_{C^1([-d,d])}=0
$$
uniformly for all $\|v\|_{C^1([-d,d])} <1,$ see [Proposition 4.8 (b), \cite{Zeidler:I}].

We obtain the estimate
\begin{equation} \label{eq:H_beta:Gauteaux}
\begin{split}
&\left\| \dfrac{H_\beta(U+tv)-H_\beta(U)}{t} -S[\beta,U]v \right\|_{C^1([-d,d])} \leq\\
&\Big( \|\omega\|_\infty+\|\omega'\|_\infty \Big) \int\limits_{-d}^{d}\left| \dfrac{f_\beta(U(y)+tv(y))-f_\beta(U(y))}{t}-f'_\beta(U(y))v(y)\right|dy.
\end{split}
\end{equation}

From $f_\beta \in C^{0,1}(\R),$ see Assumption \ref{ass:f}(5),
$$\dfrac{f_\beta(U(y)+tv(y))-f_\beta(U(y))}{t} \to f'_\beta(U(y))v(y)$$ almost everywhere on $[-d,d]$ uniformly in $v$ on $\{v:\, \|v\|_{C^1([-d,d])}<1\}$, and
$$\left|f_\beta(U(y)+tv(y))-f_\beta(U(y))\right|/|t| \leq \|f'_\beta\|_\infty \|v\|_{C^1([-d,d])}.$$
Thus, by the Lebesgue Dominated Convergence theorem the integral in \eqref{eq:H_beta:Gauteaux} converges to zero uniformly.
\end{proof}

To show the Fr\'{e}chet differentiability of the operator $H_\infty$ we must proceed in a different way due to the discontinuity of $f_\infty.$
\begin{lemma}\label{lemma:5:2}
The operator $H_\infty:B_\epsilon(U_\infty)\subset C^1_{even}([-d,d])\to C^{0,1}_{even}([-d,d])$ given in \eqref{eq:H2} is Fr\'{e}chet differentiable at $U\in B_\epsilon(U_\infty)$ with the derivative

\begin{equation} \label{eq:S:inf}
H'_\infty(U)=S[\infty,U], \quad S[\infty,U]v=\sum\limits_{i=1}^N \dfrac{\omega(x-b_i)+\omega(x+b_i)}{|u'_\infty(a_i)|}.
\end{equation}
where $b_i \in (a_i^-(\epsilon), a_i^+(\epsilon)),$ $i=1,...,N,$ are the positive solutions to $u(x)=h.$
\end{lemma}

\begin{proof}
We start by calculating the G\^{a}teaux derivative of $H_\infty$ at $U\in B_\epsilon(U_\infty).$
Consider
\begin{equation} \label{eq:H_inf:Gauteaux}
 H_\infty(U+tv)-H_\infty(U)=\int\limits_{\cO^+(t)} \omega(x-y)dy-\int\limits_{\cO^-(t)}\omega(x-y)dy
\end{equation}
where
$$
\cO^+(t)=\{y:\, U(y)\leq h \mbox{ and } U(y)+tv(y)\geq h\}
$$
 and
 $$
\cO^-(t)=\{y:\, U(y)\geq h \mbox{ and } U(y)+tv(y)\leq h\}.
$$
 Since $U\in B_\epsilon(U_\infty)$ is a restriction of a regular $(h; -b_N,\ldots,-b_1,b_1,\ldots,b_N)$ bump on $[-d,d]$ it is clear that $\pm b_i$ for all $i=1,\ldots,N,$ belongs either to $\cO^+(t)$ or $\cO^-(t).$
We remind here that $U$ and $v$ are even. Thus $\cO^\pm(t)$ are symmetric and without loss of generality we consider only $y\geq 0.$

Let $b_i,$ $i \in \{1,\ldots,N\}$ belong to either $\cO^-(t)$ or $\cO^+(t)$ and $y\not=b_i$ be a limiting point of this set, and hence $U(y)+tv(y)=h.$ By the mean value theorem
$$
y=b_i -\dfrac{tv(b_i)}{U'(\xi)+tv(\xi)}
$$
where $\xi$ lies in between of $y$ and $b_i.$ Then we get
$$
\lim \limits_{t\to 0} \dfrac{|y-b_i|}{t}=\lim\limits_{t\to 0} \dfrac{|t|}{t} \dfrac{|v(b_i)|}{|U'(b_i)|}.
$$
If $b_i \in \cO^+(t)$ then either $v(b_i)<0,$ $t>0$ or  $v(b_i)>0,$ $t<0$ and therefore
$$
\lim \limits_{t\to 0} \dfrac{|y-b_i|}{t}=\dfrac{v(b_i)}{|U'(b_i)|}, \quad b_i \in \cO^+(t).
$$
Similarly, for $b_i\in \cO^-(t)$ we conclude that
$$
\lim \limits_{t\to 0} \dfrac{|y-b_i|}{t}=-\dfrac{v(b_i)}{|U'(b_i)|}, \quad b_i \in \cO^-(t).
$$

Making use of \eqref{eq:H_inf:Gauteaux} and limits above we obtain
the G\^{a}teaux derivative of $H_\infty$ at $U\in B_\epsilon(U_\infty),$

\begin{equation*}
\begin{split}
dH_\infty(U;v)=&\lim \limits_{t\to 0} \dfrac{H_\infty(U_\infty+tv)-H_\infty(U)}{t}=\\
&\sum\limits_{\pm b_i \in \cO^+(t)\cup \cO^-(t)}\dfrac{\omega(\pm b_i-x)}{|U'(b_i)|} v(a_i)=\\
&\sum\limits_{i=1}^N \dfrac{\omega(b_i-x)+\omega(b_i+x)}{|U'(b_i)|} v(a_i)=S[\infty,U]v.
\end{split}
\end{equation*}

As $dH_\infty(U;v)$ is continuous at $U$ for any $U\in B_\epsilon(U_\infty)$ we conclude that $H_\infty$ is the Fr\'{e}chet differentiable at $U$ with $H'_\infty(U)v=dH_\infty(U;v),$ see Proposition 4.8(c) in \cite{Zeidler:I}.
\end{proof}

\begin{remark}
When $U=U_\infty$ the derivative $H'_\infty(U_\infty)$ is given as
\begin{equation} \label{eq:S:inf:2}
H'_\infty(U_\infty)=S[\infty,U_\infty], \quad S[\infty,U_\infty]v=\sum\limits_{i=1}^N \dfrac{\omega(x-a_i)+\omega(x+a_i)}{|u'_\infty(a_i)|}.
\end{equation}
\end{remark}

From Lemma \ref{lemma:5:1} and Lemma \ref{lemma:5:2} the Fr\'{e}chet derivative with respect to the second variable  at $(\beta,U)$ exists and is given as $P'_U[\beta,U]=I-S[\beta,U]$ with $S[\beta,U]$ given in \eqref{eq:S} for $\beta<\infty$ and \eqref{eq:S:inf} for $\beta=\infty.$

Now we turn to the proof of the norm convergence of $P'_U[\beta,U],$ see (iv) in Theorem \ref{th:IFT}.
As
$$
\|P'_U[\beta,U]-P'_U[\infty,U_\infty]\|_{C^{0,\alpha}([-d,d])} \leq \|U-U_\infty\|_{C^{0,\alpha}}+
\|S[\beta,U]-S[\infty,U_\infty]\|_{C^{0,\alpha}([-d,d])},
$$
it suffices to show that $\|S[\beta,U]-S[\infty,U_\infty]\|_{C^{0,\alpha}([-d,d])}$ as $(\beta,U)\to (\infty,U_\infty).$ Before we prove this operator convergence we need the following lemma.
\begin{lemma}\label{lemma:6}
Let $O(a_i) \subset (a_i^-(\epsilon),a^+(\epsilon)),$ $i=1,\ldots,N,$ be an open neighbourhood of $a_i(\epsilon)$ and $\cO$ non empty subset of $[0,d]/\bigcup_{i=1}^N O(a_i).$  Then for any $\varphi \in C([-d,d])$ we have

\begin{equation} \label{eq:conv:1}
\int\limits_{\cO} f'_{\beta_n}(U_n(x))\varphi(x)dx \to 0
\end{equation}
and
\begin{equation} \label{eq:conv:2}
\int\limits_{O(a_i)} f'_{\beta_n}(U_n(x))\varphi(x)dx \to \dfrac{\varphi(a_i)}{|u_\infty'(a_i)|}
\end{equation}
for $\beta_n\to \infty$ and $U_n \to U_\infty$ in the $C^1([-d,d])$ norm as $n\to \infty.$
\end{lemma}

\begin{proof}

We first prove \eqref{eq:conv:1}. Let $\eta<\epsilon$ be such that $[a^-_i(\eta),a_i^+(\eta)] \subset O(a_i),$ $i=1,\ldots,N.$ As $n\to \infty$  we have $U_n\in B_\eta(U_\infty)$ and thus $|U_n(x)-h|>\eta$ on the set $D_\eta:=[0,d]/\bigcup_{i=1}^N [a_i^-(\eta),a_i^+(\eta)],$ see Remark \ref{rem:3}. From Assumption \ref{ass:f}(5) and \ref{eq:f:uniformly} we have $|f'_{\beta_n}(U_n(x))|\leq C_\eta(\beta_n) \to 0$ as $n\to \infty.$ Hence, we conclude that
$$ \left|\int\limits_{\cO} f'_{\beta_n}(U_n(x))\varphi(x)dx\right|\leq
\|\varphi\|_{L^\infty([-d,d])}\int\limits_{D_\eta} |f'_{\beta_n}(U_n(x))|dx \leq d\, \|\varphi\|_{L^\infty([-d,d])} C_\eta(\beta_n)\to 0.$$

Now we turn our attention to proving \eqref{eq:conv:2}. Due to \eqref{eq:conv:1} we can without loss of generality assume that $O(a_i)=(a_i^-(\epsilon),a_i^+(\epsilon)),$ $i=1,\ldots,N.$ Let us fix some $i=1,\ldots,N.$ Due to \cite{Barvinek:1991}, $U_\infty$ and $U_n$ have the inverse functions $k_\infty(t)$ and $k_n(t),$ respectively, defined on $[h-\epsilon, h+\epsilon].$  Moreover, by the Implicit Function Theorem $\|k_n-k_\infty\|_{C^1([h-\epsilon,h+\epsilon])} \to 0$ as $n\to \infty.$ Using a change of variables and \eqref{eq:conv:1} we obtain

\begin{equation} \label{eq:conv:2:1}
\int\limits_{O(a_i)} f'_{\beta_n}(U_n(x))\varphi(x)dx -\int\limits_{h-\epsilon}^{h+\epsilon} f'_{\beta_n}(t)|k'_n(t)| \varphi(k_n(t))dt \to 0, \quad n\to \infty.
\end{equation}

We note that for any positive $\eta<\epsilon$
\begin{equation}\label{eq:conv:aux}
\int\limits_{h-\eta}^{h+\eta} f'_{\beta_n}(t)|k'_\infty(h)| \varphi(k_\infty(h))dt
=\dfrac{\varphi(a_i)}{|u'_\infty(a_i)|}
\end{equation}
and
\begin{equation}\label{eq:I(eta;n)}
I(\eta,n)=\int\limits_{h-\eta}^{h+\eta} f'_{\beta_n}(t)\Big(|k'_n(t)| \varphi(k_n(t))-k'_\infty(h)\varphi(k_\infty(h))\Big)dt.
\end{equation}

Let $I(\eta,n)=I_1(\eta,n)+I_2(\eta,n)$ where
$$I_1(\eta,n)=\int\limits_{h-\epsilon}^{h+\epsilon} f'_{\beta_n}(t)\Big(|k'_n(t)| \varphi(k_n(t))-|k'_\infty(t)| \varphi(k_\infty(t))\Big)dt$$
and
$$I_2(\eta,n)=\int\limits_{h-\epsilon}^{h+\epsilon} f'_{\beta_n}(t)\Big(|k'_\infty(t)| \varphi(k_\infty(t))-|k'_\infty(h)| \varphi(k_\infty(h))\Big)dt.$$

We observe that $|I_1(\eta,n)|\to 0$ and  $|I_2(\eta,n)|\to o(\eta)$  as $n\to \infty$ since

$$|I_1(\eta,n)| \leq \max \limits_{t\in [h-\eta,h+\eta]} \Big( \varphi(k_n(t)) |k'_n(t)|-\varphi(k_\infty(t)) |k'_\infty(t)|\Big)\to 0$$
due to  $\|k_n - k_\infty\|_{C^1([-d,d])} \to 0$  and
$$|I_2(\eta,n)| \leq \max \limits_{t\in [h-\eta,h+\eta]} \Big( \varphi(k_\infty(t)) |k'_\infty(t)|-\varphi(k_\infty(h)) |k'_\infty(h)|\Big)\to o(\eta).$$

Thus, we obtain
$$
\lim \limits_{n\to 0}|I(\eta,n)|\leq \lim \limits_{n\to 0}|I_1(\eta,n)|+\lim \limits_{n\to 0}|I_2(\eta,n)| =o(\eta)
$$
that combined with \eqref{eq:I(eta;n)} and \eqref{eq:conv:aux} results in

\begin{equation} \label{eq:conv:2:2}
\lim \limits_{n\to 0}\int \limits_{h-\eta}^{h+\eta} f'_{\beta_n}(t)|k'_n(t)| \varphi(k_n(t))dt = \dfrac{\varphi(a_i)}{|u'_\infty(a_i)|}+ o(\eta).
\end{equation}

From \eqref{eq:conv:1} we get
$$
\int \limits_{a_i^-(\epsilon)}^{a_i^+(\epsilon)} f'_{\beta_n}(U_n(x))\varphi(x)dx-\int \limits_{a_i^-(\eta)}^{a_i^+(\eta)} f'_{\beta_n}(U_n(x))\varphi(x)dx \to 0, \quad n\to \infty,
$$
 independently on $\eta.$ This observation and \eqref{eq:conv:2:1} yields $o(\eta)\equiv 0$ in \eqref{eq:conv:2:2} which completes the proof.
\end{proof}

\begin{proposition} \label{prop:conv}
Let Assumption \ref{ass:f}, Assumption \ref{ass:omega} and Assumption \ref{ass:u:infty} be satisfied and
$U_\infty=T_1u_\infty$ be a restriction of $u_\infty$ on $[-d,d].$. Then the map  $(0,\infty] \times B_\epsilon(U_\infty)\mapsto S[\beta,U]$ is continuous at $(\infty, U_\infty),$ that is,
$$\|S[\beta,U] - S[\infty, U_\infty]\|_{C^{0,\alpha}([-d,d])}\to 0$$
as $\beta \to \infty$ and $\|U-U_\infty\|_{C^1([-d,d])}\to 0.$
\end{proposition}

\begin{proof}
Let $\beta<\infty$, $\beta_n \to \infty$ and $\|U-U_\infty\|_{C^1([-d,d])}\to 0$ as $n\to \infty.$
We introduce the partition of $[0,d],$ $0=y_0<y_1<...<y_{N}=d$ where $y_i= (a_{i}+a_{i-1})/2,$ $i=1,\ldots,N$
and represent $S[\beta_n,U_n]$ as
$$
S[\beta_n,U_n]=\sum\limits_{i=1}^N (S_i^+(n)+S_i^-(n))
$$
with
$$
S_i^\pm(n)v= \pm \int \limits_{\pm y_i}^{ \pm y_{i+1}} \omega(x-y)f'_{\beta_n}(U_n(y))v(y)dy.
$$

For $\beta=\infty$ we have
$$
S[\infty,U_n]=\sum\limits_{i=1}^N (S_i^{+}+S_i^{-}),
\quad
S_i^\pm v= \dfrac{\omega(x\pm a_i)}{|u'_\infty(a_i)|}v(a_i).
$$

Next we prove that $\|S_i^{+}(n)-S_i^+\|_{C^{0,\alpha}([-d,d])} \to 0,$ $n\to 0,$ for any $i=1,...,N.$

For an arbitrary fixed $i \in \{1,\ldots,N\}$

$$(S_i^+(n)v)(x)-(S_i^+v)(x)=g_1(x,n)+g_2(x,n)$$
where
$$
g_1(x,n)=\int \limits_{y_i}^{ y_{i+1}} (\omega(x-y)-\omega(x-a_i))f'_{\beta_n}(U_n(y))v(y)dy
$$
and
$$
g_2(x,n)=\omega(x-a_i)\int \limits_{y_i}^{ y_{i+1}} f'_{\beta_n}(U_n(y))\Big(v(y)-\dfrac{v(a_i)}{u'_\infty(a_i)}\Big)dy.
$$

Next we show that $\|g_k(\cdot,n)\|_{\alpha}=o(1/n)\|v\|_\alpha$, $k=1,2,$ as $n\to \infty.$
We start with $g_1(x,n).$ Let $x\in[y_i,y_{i+1}]$ then
\begin{equation*}
\begin{split}
|g_1(x,n)|\leq &\int \limits_{y_i}^{ y_{i+1}} |\omega(x-y)-\omega(x-a_i)| f'_{\beta_n}(U_n(y)) |v(y)|dy \leq \\
& L \|v\|_{\alpha}\int\limits_{y_i}^{ y_{i+1}} |y-a_i| f'_{\beta_n}(U_n(y))dy
\end{split}
\end{equation*}
with $L\leq\|\omega'\|_{\infty}$ being the Lipschitz constant of $\omega(x)$ on the interval $[-d-y_{i+1},d+y_{i+1}].$

Using \eqref{eq:conv:1} with $\varphi(y)=|y-a_i|$ the integral on the right hand side tends to zero as $n\to \infty.$
In order to bound the H\"{o}lder constant of $g_1(x,n)$ we use two different estimates
$$
|\omega(x-y)-\omega(x-a_i)-\omega(z-y)+\omega(z-a_i)|\leq 2 L |y-a_i|
$$
and
$$
|\omega(x-y)-\omega(x-a_i)-\omega(z-y)+\omega(z-a_i)|\leq 2 L |x-z|
$$
so that
$$
|\omega(x-y)-\omega(x-a_i)-\omega(z-y)+\omega(z-a_i)|\leq 2L |y-a_i|^{1-\alpha}|x-z|^\alpha.
$$

Thus for $x\not=z$ we have
\begin{equation*}
\begin{split}%
\dfrac{|g_1(x,n)-g_1(y,n)|}{|x-z|^\alpha}\leq 2 L \| v\|_\alpha \int \limits_{y_i}^{ y_{i+1}} |y-a_i|^{1-\alpha} f'_{\beta_n}(U_n(y)) dy.
\end{split}
\end{equation*}
Again by \eqref{eq:conv:1} the integral on the right hand side tends to zero as $n\to 0.$

To handle the term $g_2(x,n)$ we represent it as a sum $g_2(x,n)=\omega(x-a_i)(I_1(n)-I_2(n))$
with
$$
I_1(n)=\int\limits_{y_i}^{y_{i+1}} f'_{\beta_n}(U_n(y))(v(y)-v(a_i))dy
$$
and
$$
I_2(n)=v(a_i) \left(\int\limits_{y_i}^{y_{i+1}} f'_{\beta_n} (U_n(y))dy -\dfrac{1}{|u'_\infty(a_i)|}\right).
$$

We have the estimates
$$
|I_1(n) \leq \|v\|_{\alpha} \int\limits_{y_i}^{y_{i+1}} f'_{\beta_n}(U_n(y))|y-a_i|dy
$$
and
$$
|I_2(n) \leq \|v\|_{\alpha}\left( \int\limits_{y_i}^{y_{i+1}} f'_{\beta_n}(U_n(y))dy- \dfrac{1}{|u'_\infty(a_i)|}\right),
$$
where from Lemma \ref{lemma:6} both integrals on the right hand sides tend to zero as $n\to 0.$

It follows that for any $x\in(y_i,y_{i+1}),$
$$
|g_2(x,n)|\leq \|\omega\|_\infty (|I_1(n)|+I_2(n))=o(1/n) \|v\|_\alpha, \mbox{ as } n\to 0,
$$
and
\begin{equation*}
\begin{split}
\dfrac{|g_2(x,n)-g_2(z,n)|}{|x-z|^\alpha}&\leq
\sup\limits_{
\begin{array}{c}
x,z\in[y_i,y_{i+1}]\\
x\not=z
\end{array}
}
\dfrac{|\omega(x-a_i)-\omega(z-a_i)|}{|x-z|^\alpha}
 \| v\|_\alpha (|I_1(n)|+I_2(n))\leq \\
 &L \sup \limits_{x,z} |x-z|^{1-\alpha} (|I_1(n)|+|I_2(n)|)<(d)^{1-\alpha} L \|v\|_\alpha o(1/n), n\to \infty.
\end{split}
\end{equation*}
Collecting the estimates for $\|g_1(x,n)\|_\alpha$ and $\|g_2(x,n)\|_\alpha$ we conclude that
$\|S^+_i(n)-S^+_i\|_\alpha=o(1/n),$ $n\to \infty.$
Due to the symmetry $\|S^-_i(n)-S^-_i\|_\alpha=o(1/n)$ as well.
\end{proof}

\begin{remark}\label{rem:H-space vs Lip}
Notice that the need for $C^{0,\alpha}([-d,d])$ space with $0<\alpha<1$ comes when estimating the convergence of $|g_i(x,n)-g_i(z,n)|/|x-z|^\alpha,$  $n\to 0,$ $i=1,2,$ as similar arguments would fail for $\alpha=1.$
\end{remark}

With the next proposition we prove that the condition (iv) of Theorem \ref{th:IFT} is satisfied.
\begin{proposition}
Let $u_\infty$ be a symmetric bump solution to \eqref{eq:H}, $U_\infty=T_1u_\infty$ be its restriction on $[-d,d],$ and assume that  Assumptions \ref{ass:omega} and Assumption \ref{ass:u:infty} are satisfied. Then
the Fr\'{e}chet derivative of $P(\infty,U)$ at $U_\infty,$  $P'[\infty,U_\infty]=I-S[\infty,U_\infty]:C^1_{even}([-d,d]) \to C^{0,\alpha}([-d,d])$ with $S[\infty,U_\infty]$ defined in \eqref{eq:S:inf} is invertible.
\end{proposition}

\begin{proof}
We will show that Assumption \ref{ass:u:infty}(2) implies $P'[\infty,U_\infty]$ has no zero eigenvalue.
Fist we notice that the operator $S[\infty, U_\infty]$ has the same eigenvalues as the $N\times N$ matrix ${\textsf S}=(S_{ij})$
\begin{equation*}
S_{ij}=\frac{\omega(a_i-a_j)+\omega(a_i+a_j)}{|u'_\infty(a_j)|}, \quad i,j=1,\ldots,N.
\end{equation*}

Introduce the real diagonal matrix ${\textsf C}:=diag(|u'_\infty(a_1)|,\ldots,|u'_\infty(a_N)|).$ The
 matrix ${\textsf S}$ can be made symmetric as ${\textsf C}^{1/2}{\textsf S} {\textsf C}^{-1/2}$ and thus has only real eigenvalues. The operator $P'[\infty,U_\infty]$ has in turn the same eigenvalues as the matrix ${\textsf P}:={\textsf I}-{\textsf S}$ with ${\textsf I}$ being the identity matrix. We notice that the elements of ${\textsf P}=(P_{ij}),$ $i,j=1,\ldots,N$ are
 \begin{equation}
 \begin{array}{l}
  P_{ii}=\dfrac{|u'_\infty(a_j)|-\omega(0)-\omega(2a_i)}{|u'_\infty(a_i)|}=
  \dfrac{1}{|u'(a_i)|}\dfrac{\partial g_i}{\partial a_i}({\bf a}), \\
  \\
  P_{ij}=-\dfrac{\omega(a_i-a_j)+\omega(a_i+a_j)}{|u'_\infty(a_j)|}= \dfrac{(-1)^{i+j}}{|u'(a_i)|}\dfrac{\partial g_i}{\partial a_j}({\bf a}), \quad i\not=j,
 \end{array}
  \end{equation}
with $\partial g_i/\partial a_j({\bf a})$ defined in Assumption \ref{ass:u:infty}(2).

 Let $\textsf{J}$ be the Jacobian matrix defined in Assumption \ref{ass:u:infty}(2), $\textsf{D}:=diag((-1)^1,\ldots,(-1)^N),$ and $\textsf{C}$ given as above. Then we have $\textsf{P} \sim \textsf{PC} = \textsf{DJD} \sim \textsf{J}.$ As $\det(\textsf{J})\not=0$ we conclude that the matrix $\textsf{P}$ has no zero eigenvalue as well as the operator $P'[\infty,U_\infty].$
\end{proof}
Thus, we are ready to apply Theorem \ref{th:IFT} for the operator $P$ in \eqref{eq:op:P}.

\subsection*{Step 2}
  Theorem \ref{th:IFT} (a)-(c)  yields  the existence and uniqueness of the fixed point $U_\beta \in B_\epsilon (U_\infty)$ of the operator $H_\beta$ for large $\beta>0$ and the convergence $\|U_\beta-U_\infty\|_{C^1([-d,d])}\to 0$ as $\beta\to \infty$.

 Next, we apply the reconstruction operator $T_2$ given in \eqref{eq:T2} to $U_\beta,$ that is, $u_\beta=T_2U_\beta.$ From Proposition \ref{prop:1} $u_\beta$ is then a N-bump solution to \eqref{eq:H}. Let $\beta \to \infty,$ from the estimate
 $$ \|u_\beta-u_\infty\|_{C^1} \leq \left(\|\omega\|_\infty +\|\omega'\|_\infty\right) \|N(\beta,U_\beta)-N(\infty,U_\infty)\|_{L^1([-d,d])},$$
Lemma \ref{lemma:4}, and $\|U_\beta-U_\infty\|_{C^1([-d,d])}\to 0$ the convergence $\|u_\beta-u_\infty\|_{C^1} \to 0$ follows.

\subsection*{Step 3}

From the second part of Theorem \ref{th:IFT} the fixed point $U_\beta\in B_\epsilon(U_\infty)$ of $H_\beta$ for sufficiently large $\beta$ can be obtain by the sequence of successive approximations $\{U_n\}$, $\|U_n-U_\beta\|_{C^1([-d,d])} \to 0,$ $n\to \infty,$ as
$$
U_{n+1}=U_n-(P'_U[\infty,U_\infty])^{-1}\circ P(\beta,U_n), \quad U_0=U_\infty,
$$

or, using \eqref{eq:op:P} and \eqref{eq:S:inf:2}
\begin{equation}\label{eq:approx:U_n}
U_{n+1}=U_n-(I-S[\infty,U_\infty])^{-1}\circ (I-H_\beta(U_n)), \quad U_0=U_\infty.
\end{equation}

First we obtain the expression for $(I-S[\infty,U_\infty])^{-1}.$
Let $(I-S[\infty,U_\infty]) v=w$ and introduce ${\bf v}=(v(a_1),\ldots,v(a_N))^T$ and ${\bf w}=(w(a_1),\ldots,w(a_N))^T.$ Then ${\bf v}=(\textsf{I}-\textsf{S})^{-1} {\bf w}$ and
$v(x)=w(x)+S[\infty,U_\infty]v(x)=w(x)+{\bf s}^T \bf{ v}
$
where $\textsf{I},$ ${\bf s}$ and $\textsf{S}$ are as in Theorem \ref{th:1}.

From the last two formulae we derive
\begin{equation} \label{eq:P'inv}
v(x)=w(x)+ {\bf s}^T(\textsf{I}-\textsf{S})^{-1} {\bf w}.
\end{equation}

Combining \ref{eq:P'inv} and  \eqref{eq:approx:U_n} we arrive at \eqref{eq:u-appx}.

On each iteration step we can obtain the corresponding $u_n=T_2U_n$ with $T_2$ given in \eqref{eq:T2}.
It remains to notice that
$$\|u_n-u_\beta\|_{C^1}=\|T_2 U_n-T_2U_\beta\|_{C^1}\leq (\|\omega\|_\infty+\|\omega'\|_\infty)\|N(\beta,U_n)-N(\beta,U_\beta)\|_{L^1([-d,d])}\to 0 $$
as $n\to \infty$ due to Lemma \ref{lemma:3}.

\section{Advantages for numerical construction} \label{sec:Conclusions}

In this section we apply Theorem \ref{th:1} to demonstrate the existence of $1$-bump solutions of the FitzHugh-Nagumo equation \eqref{eq:FitzHugh:2} and $2$-bump solutions to the Amari model \eqref{eq:1} with $\omega$ given in \eqref{eq:omega:2Gaussians} and $f_\beta$  as in \eqref{eq:f:hill}. We also compute the approximations of the bump solutions using \eqref{eq:u-appx} and discuss the advantages of this numerical approximation compared to other approaches.

\subsection{$1$-bump solution of FitzHugh Nagumo equation}\label{sec:Conclusions:FN}

Let us investigate the equation \eqref{eq:FN:steady:a} with $f_\beta$ given by \eqref{eq:f:hill} with $p=2$ and its corresponding integral equation \eqref{eq:FN:steady:int}, for the existence and numerical construction of bump solutions.
A bump solution corresponds to the homoclinic orbit in the phase plane of the equation
\begin{equation} \label{eq:odes}
\begin{array}{l}
u'=v\\
v'=k^2u-f_\beta(u)
\end{array}
\end{equation}
which exists when $k<1/\sqrt{2h}.$
Any bounded solution to \eqref{eq:odes} is confined in the closure of the bounded open region $D$ of the phase plane with the homoclinic orbit being its boundary $\partial D$, see e.g. Fig.\ref{fig:ode:FN}a.

\begin{figure}[h]
\scalebox{0.45}{\includegraphics{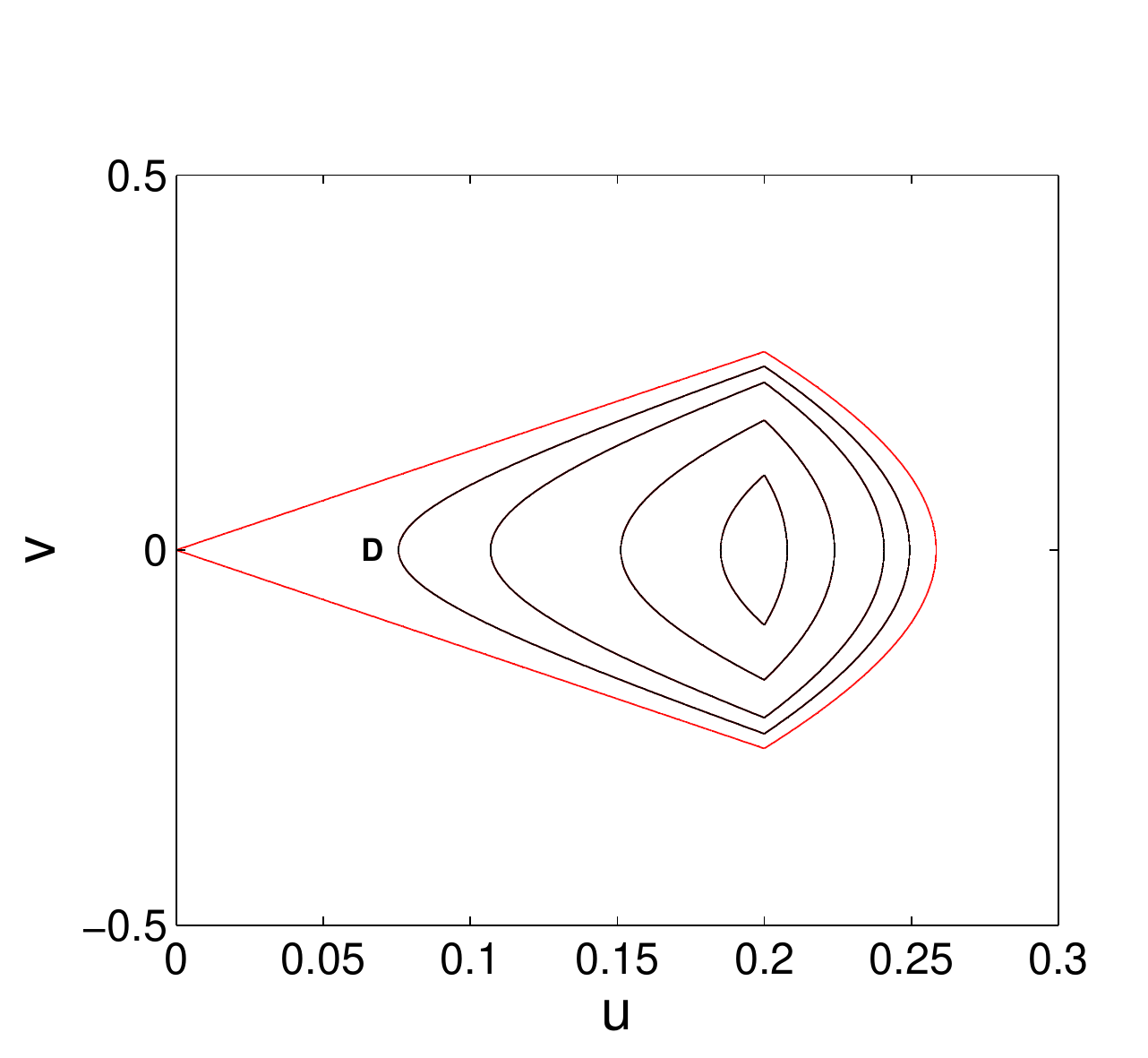}}
\scalebox{0.45}{\includegraphics{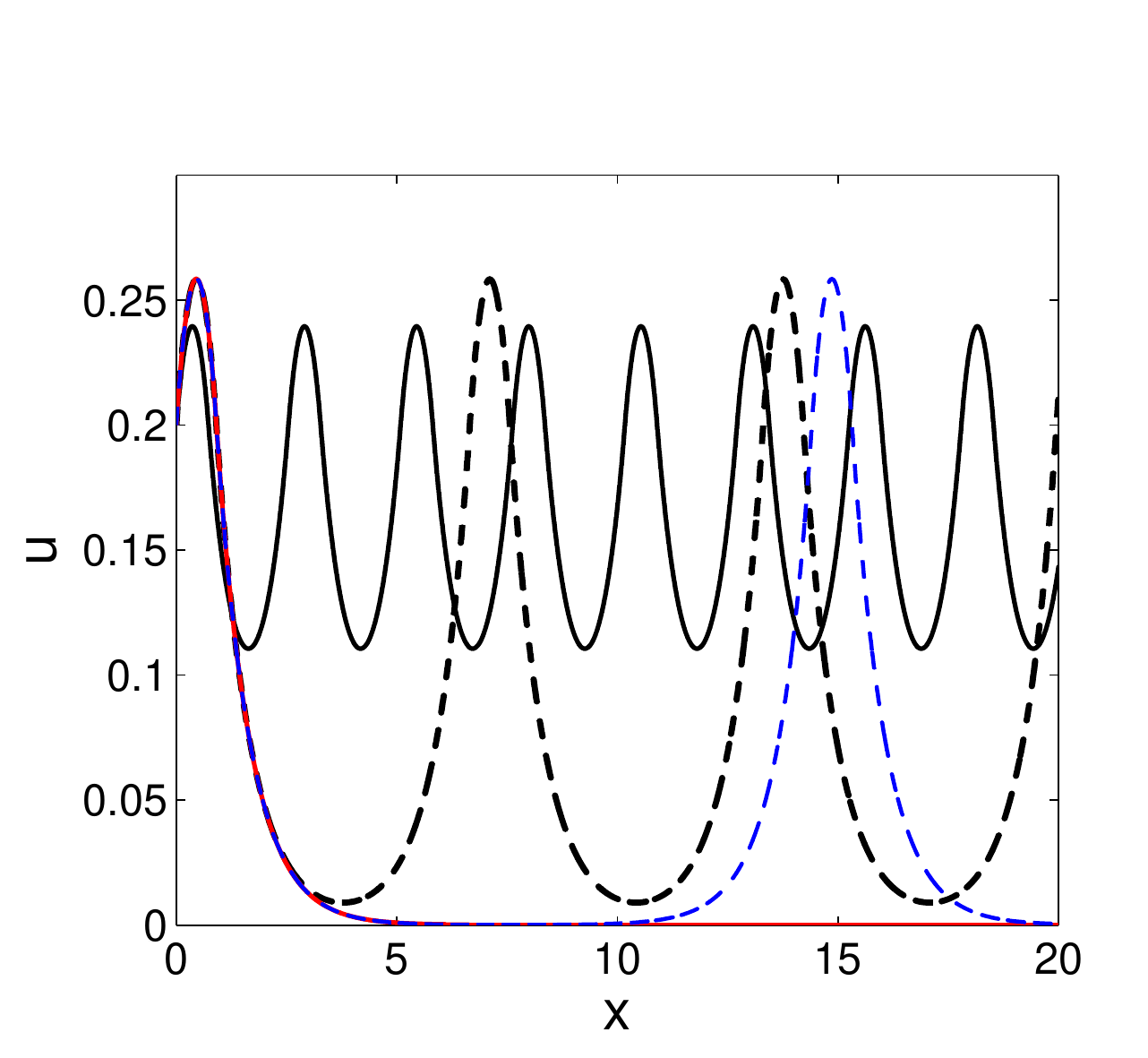}}
\caption{For $\beta=\infty,$ $k=1.339,$ and $h=0.2$ we plot (a) periodic orbits (black) and the homoclinic orbit (red) of \eqref{eq:odes} and (b) periodic solutions (black and blue) and the bump solution (red) to \eqref{eq:FN:steady:a} } \label{fig:ode:FN}
\end{figure}
Notice that the system \eqref{eq:odes} is reversible and conservative.
From the reversibility it follows that any solution of \eqref{eq:odes} with the initial conditions in $D$ is a periodic orbit and thus periodic orbits are dense in $D,$ see Fig. \ref{fig:ode:FN}a. This fact causes some difficulties to obtain the homoclinic orbit numerically. We have plotted in Fig.\ref{fig:ode:FN}b the bump solution obtained analytically (red line) and by solving \eqref{eq:odes} numerically with $(u(0),v(0))=(h,kh)$ (blue dashed line). This illustrates that in order to obtain a good approximation of the bump solution on a large interval using the shooting method one must increase the precision of the method accordingly. Moreover, the shooting method would not be straightforwardly applicable if $f_\beta$ does not satisfy Assumption \ref{ass:f}(2). This supports our argument for replacing the cubic $g(u)$ in \eqref{eq:FitzHugh} with $-u+f_\beta,$ see Section \ref{sec:Examples:FN}.

Now we analyse the equation \eqref{eq:FN:steady:int} where Assumption \ref{ass:f} and Assumption \ref{ass:omega} are satisfied. For $\beta=\infty$ one can obtain an explicit formula for the $(h;a)$-bump solution $u_\infty$ to \eqref{eq:FN:steady:int} using the Amari technique \cite{Amari:1977} or by solving the ordinary differential equation \eqref{eq:FN:steady:a}. We calculate $a=-\ln(1-2k^2h)/2k$ and $u'_\infty(a)=kh>0,$ that is the bump is regular and Assumption \ref{ass:u:infty}(1) is satisfied. The condition (2) of Assumption \ref{ass:u:infty} is reduced to $\omega(0)\not=0$ which is also fulfilled. Thus, by Theorem \ref{th:1} there exists a regular $1$-bump solution to \eqref{eq:FN:steady:int} that converges in $C^1(\R)$-norm to $u_\infty$ and can be constructed using the iteration scheme \eqref{eq:u-appx}. In Fig.\ref{fig:Iterations:FN}a we have plotted the approximation of the restriction of bump solution, that is, $U_\beta.$ In Fig. \ref{fig:Iterations:FN}b we have plotted the base $10$ logarithm of the relative error

\begin{equation}\label{eq:error:FN}
error(n+1)=\dfrac{\|U_{n+1}-U_n\|_{C^1([-d,d])}}{\|U_n\|_{C^1([-d,d])}}, \quad n=0,1,\ldots .
\end{equation}
where $U_n$ is given in \eqref{eq:u-appx:b}.

\begin{figure}[h]
\scalebox{0.45}{\includegraphics{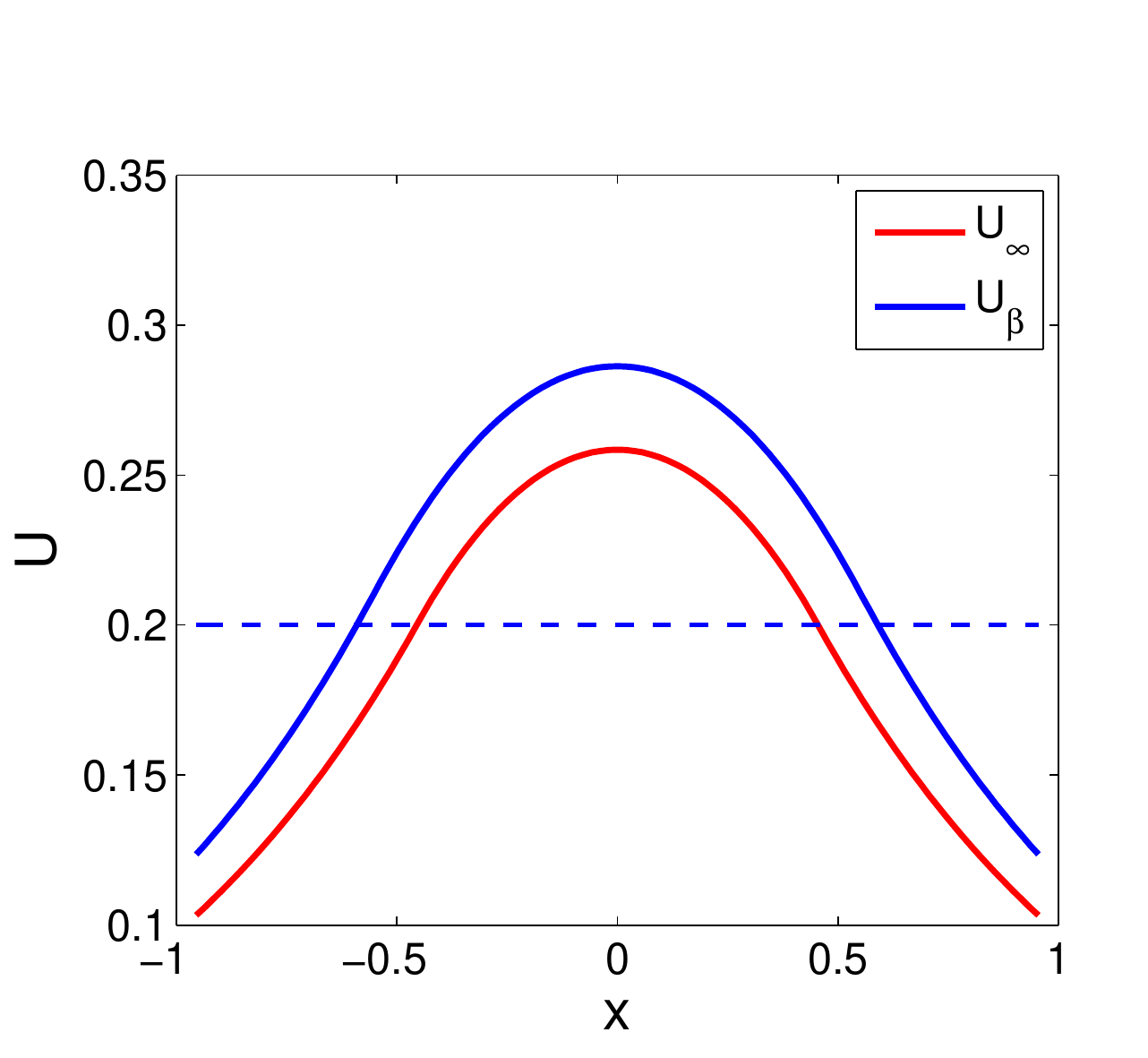}}
\scalebox{0.45}{\includegraphics{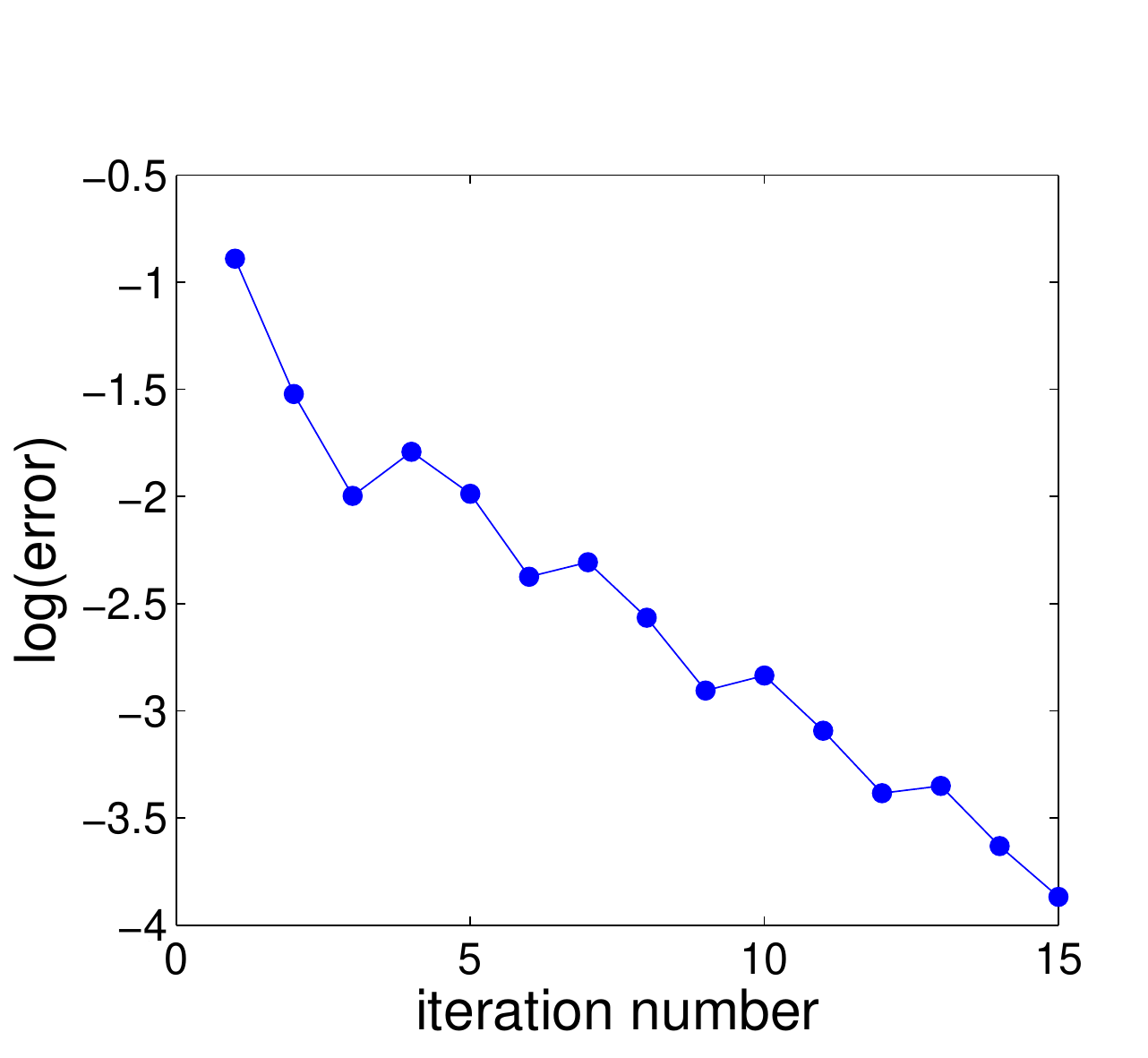}}
\caption{(a) Approximation of the bump solution to \eqref{eq:FN:steady:int} on $[-d,d]$ with $d=a+0.5$ and $\beta=100$ after the 15th iteration. (b) The logarithm of the relative error \eqref{eq:error:FN} of the iteration process. } \label{fig:Iterations:FN}
\end{figure}

Notice, that despite the bump solution $u_\infty$ is non-isolated fixed point of \eqref{eq:H} with $\beta=\infty$ on $C^1(\R)$, it is isolated in $K_\epsilon(u_\infty).$ Consequently, $U_\beta=T_1 u_\beta$ is an isolated fixed point of the operator $H_\beta$ already in the whole space $C^1([-d,d]).$  Thus, the proposed method allows us to isolate solutions which can be extremely useful when dealing with higher order ordinary differential equations.

\subsection{2-bump solutions of the neural field model} \label{sec:Conclusions:NFM}

In this section we illustrate the obtained result for the Amari model with the firing rate function $f_\beta,$ $\beta \in (0,\infty]$ as in \eqref{eq:f:hill}, $p=2$ and the connectivity function $\omega$ given as in \eqref{eq:omega:2Gaussians} with $K=3,$ $k=2,$ $M=1,$ and $m=0.5,$ see Fig.\ref{fig:omegas}. Hence, Assumption \ref{ass:f} and Assumption \ref{ass:omega} are satisfied. When $\beta=\infty$ one can employ the Amari technique \cite{Amari:1977,Murdock:2006} to show that there exist $2$-bump solutions. In particular, for $h=0.3$ there is a pair of $2$-bump solutions $u^{(1)}_\infty$ and $u^{(2)}_\infty$: the first one is the $(h; -a_2,-a_1,a_1,a_2)$-bump with $a_1=0.2948,$  $a_2=0.8506$ and the second one is with $a_1=0.3786,$ $a_2=0.6782.$ We verified numerically that Assumption \ref{ass:u:infty} is fulfilled for both bump solutions, thus Theorem \ref{th:1} can be applied.

In Fig.\ref{fig:Iterations:NFM}a we have plotted  the restriction of  $u_\infty^{(i)}$, $i=1,2,$ on $[-1,1]$
and the restriction of the approximation of the bump solutions of \eqref{eq:H} with $\beta=100.$

For both cases we have computed the relative error using \eqref{eq:error:FN} with $U_n=U_n^{(i)}$ and have plotted the base $10$ logarithm of the error in Fig. \ref{fig:Iterations:NFM}b.

\begin{figure}[h]
\scalebox{0.45}{\includegraphics{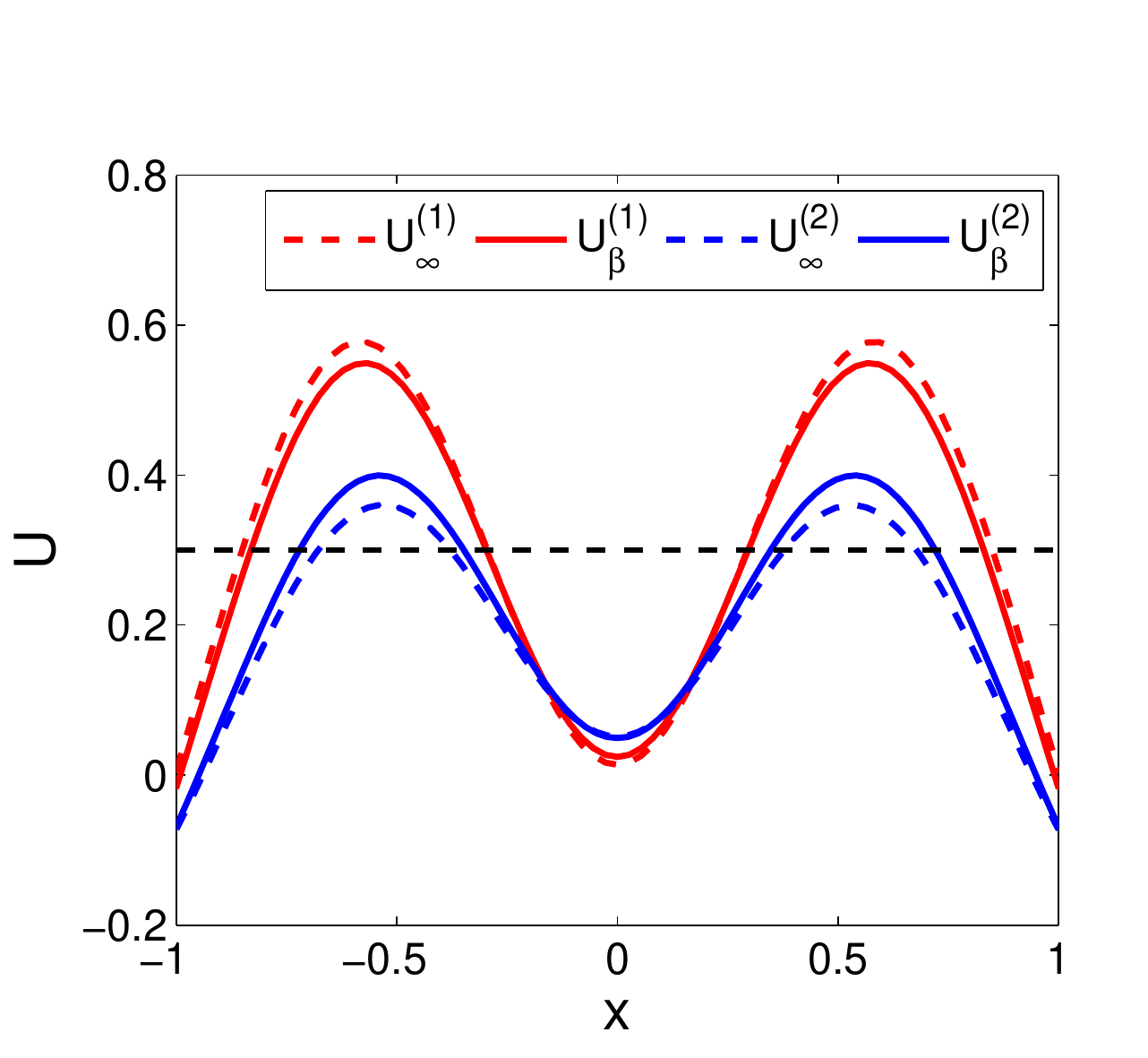}}
\scalebox{0.45}{\includegraphics{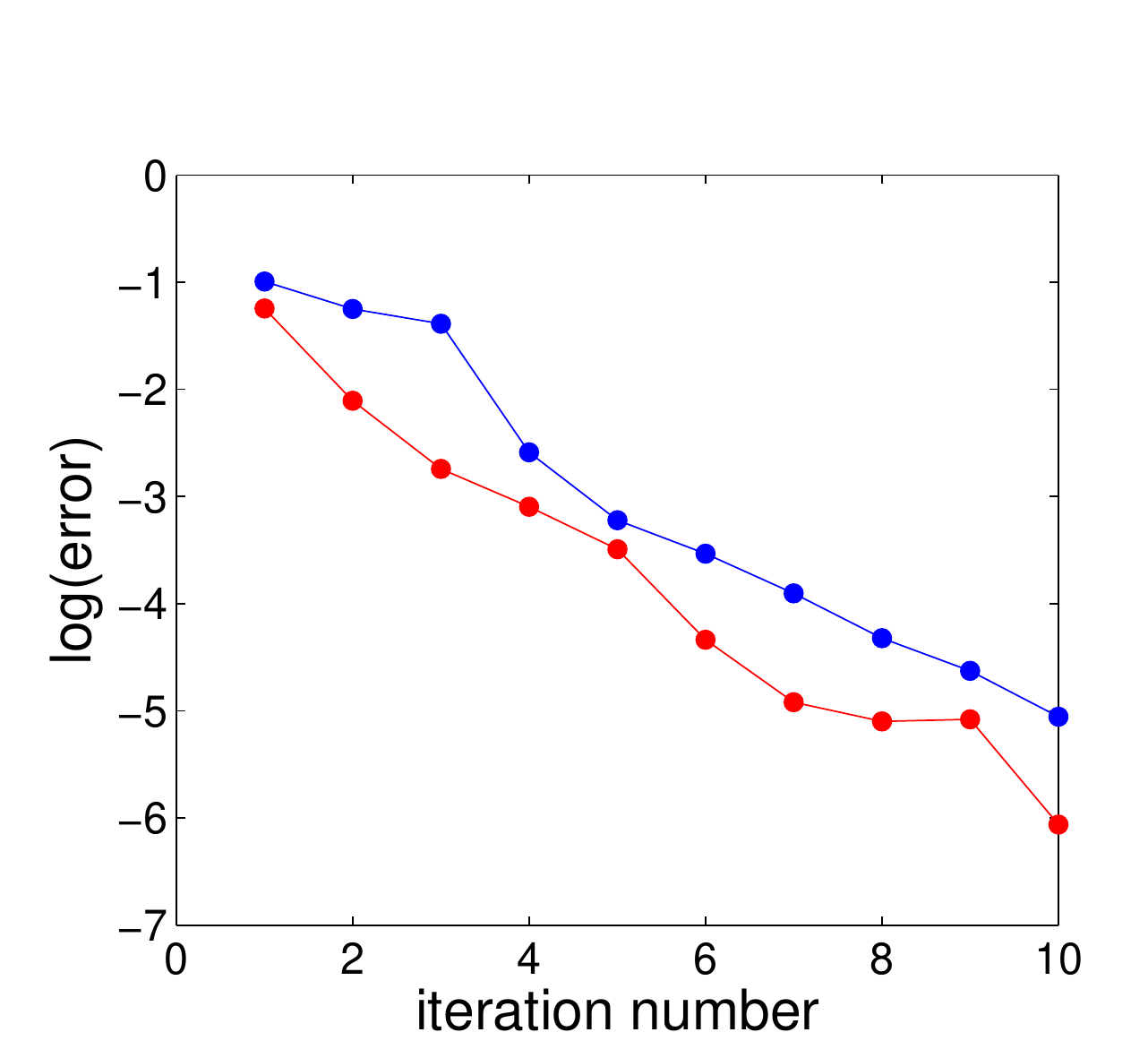}}
\caption{(a) Approximation of the $2$-bump solutions to \eqref{eq:FN:steady:int} on $[-d,d]$ with $d=a_2+0.5$ and $\beta=100$ after the $10$th iteration. (b) The logarithm of the relative error \eqref{eq:error:FN} of the iteration process. } \label{fig:Iterations:NFM}
\end{figure}

 As the Fourier transform of $\omega(x)$ is not a real rational function, the ordinary differential methods cannot be used here. Yet, when $f_\beta$ is such that $\chi_{[h+1/\beta,\infty)}\leq f_\beta<\chi_{[h,\infty)}$ it might be possible to iteratively construct $u_\beta$ using the theory of monotone operators in Banach spaces similarly as it has been done for $1$-bump solutions in \cite{Oleynik:2015}. However,this method can be tricky (or not even possible) to use for $N$-bumps solutions when $N$ is large. Moreover, some of the bump solutions cannot be obtained by this method as e.g. already for the case of $1$-bumps, the limiting solution $u_\infty$ is required to be linearly stable. As one of the $2$-bump solutions is unstable, see \cite{Laing:2002}, it is doubtful the iterative method as in \cite{Oleynik:2015} can be successful.

\section{Conclusions and Outlook}\label{sec:Outlook}
To summarize, we would like to emphasize few important points: (i)
With Theorem \ref{th:1} we justify the approximation of a smooth sigmoid function by the discontinuous unit step function for \eqref{eq:H} on the class of bump solutions. In particular, one can be assured that a bump solution $u_\beta$ for large $\beta$  exists and can be approximated by $u_\infty.$
(ii) Theorem \ref{th:1} does not require $\omega$ to be smooth nor to have a real rational Fourier transform. (iii)
In order to obtain a better approximation of $u_\beta$ than $u_\infty$ one may utilize the iteration scheme in Theorem \ref{th:1}.
Compared to the ordinary differential methods it allows us to isolate the solution and thus does not require a high numerical precision to secure that the found solution is indeed localized. However it might not be as efficient as the shooting method.

The technique presented in this paper could be fruitful in more general situations, for instance, to study existence and uniqueness of spatially localized solutions in two and three dimensional neural field models. However, the most serious obstacle to develop such a theory is the absence of a general scheme for studying bumps in the limit (discontinuous) case, the only exception being the theory of radially symmetric bumps \cite{Owen:2007,Burlakov:2016}.

Within the framework of the neural field model \eqref{eq:1}, the next step would be to study Lyapunov stability of found bumps (the work under preparation). This can be done using the properties of spectral asymptotic that follows from the norm convergence of the Fr\'{e}chet derivatives, see Proposition \ref{prop:conv}.


\bibliographystyle{unsrt}
\bibliography{ExistenceIFT_OPKS.bib}

\end{document}